\documentclass{amsart}

\usepackage{graphicx,
            mathtools,
            amsmath,
            amsfonts,
            amsthm,
            amssymb,
            bm,
            varioref}

\usepackage[margin=1in,footskip=0.25in]{geometry}

\usepackage[capitalize,noabbrev]{cleveref}
\usepackage[protrusion=false]{microtype}

\usepackage[numeric,compressed-cites]{amsrefs}

\begin{document}


\title[Proof of the Tijdeman-Zagier Conjecture via Slope Irrationality and Term Coprimality]
      {Proof of the Tijdeman-Zagier Conjecture via\\Slope Irrationality and Term Coprimality}

\author{David Hauser}
\address{Genpact, New York 10036}
\email{david.hauser@genpact.com}
\thanks{Special thanks to emeritus Professor Harry Hauser
        for support and guidance on this research.}

\author{Ian Hauser}
\address{Solomon Schechter School of Long Island, Williston Park, New York 11596}
\curraddr{Binghamton University, Binghamton, New York, USA 13902}
\email{ihauser2@binghamton.edu}



\begin{abstract}
The Tijdeman-Zagier conjecture states no integer solution exists for $A^X+B^Y=C^Z$ with positive integer bases and integer exponents greater than 2 unless gcd$(A,B,C)>1$.  Any set of values that satisfy the conjecture correspond to a lattice point on a Cartesian graph which subtends a line in multi-dimensional space with the origin.  Properties of the slopes of these lines in each plane are established as a function of coprimality of terms, such as irrationality, which enable us to explicitly prove the conjecture by contradiction.
\end{abstract}

\maketitle


\newtheorem{theorem}{Theorem}[section]
\newtheorem{lemma}{Lemma}[section]
\newtheorem{prop}{Proposition}[section]

\newtheorem{definition}{Definition}[section]

\setlength{\emergencystretch}{10pt}
\sloppy
\raggedbottom
\newcommand\Line{\noindent\rule{\textwidth}{0.4pt}}

\newcommand\BealsEq{$A^X+B^Y=C^Z$}
\newcommand\Conjecture{Given positive integers $A$, $B$, $C$, $X$, $Y$, $Z$,
                       where $X,Y,Z\geq 3$ and \BealsEq}

\newcommand\PList{p_1 p_2 p_3 \cdots p_m}
\newcommand\QList{q_1 q_2 q_3 \cdots q_n}
\newcommand\RList{r_1 r_2 r_3 \cdots r_s}

\newcommand\PAB{p_{_{AB}}}
\newcommand\PAC{p_{_{AC}}}
\newcommand\PBC{p_{_{BC}}}

\newcommand\MCB{m_{_{CB}}}
\newcommand\MCA{m_{_{CA}}}
\newcommand\MBA{m_{_{BA}}}

\newcommand\alphaCA{\alpha_{_{CA}}}
\newcommand\alphaCB{\alpha_{_{CB}}}

\newcommand\betaCA{\beta_{_{CA}}}
\newcommand\betaCB{\beta_{_{CB}}}

\newcommand{\xdownarrow}[1]{%
  {\left\downarrow\vbox to #1{}\right.\kern-\nulldelimiterspace}
}


\section{Introduction}
The Tijdeman-Zagier conjecture \cites{elkies2007abc, waldschmidt2004open, crandall2006prime}, or more formally the generalized Fermat equation \cite{bennett2015generalized} states that given
\begin{equation}
A^X+B^Y=C^Z \label{eqn:1}
\end{equation}
where $A$, $B$, $C$, $X$, $Y$, $Z$ are positive integers, and exponents $X$, $Y$, $Z\geq3$, then bases $A$, $B$, and $C$ have a common factor.  Some researchers refer to this conjecture as Beal's conjecture \cite{beal1997generalization}.  There are considerable theoretical foundational advances in and around this topic \cites{nitaj1995conjecture, darmon1995equations, vega2020complexity}.  Many exhaustive searches within limited domains have produced indications this conjecture may be correct \cites{beal1997generalization, norvig2010beal, durango}.  More formal attempts to explicitly prove or counter-prove the conjecture abound in the literature but are often unable to entirely generalize due to difficulty in overcoming floating point arithmetic limits, circular logic, unsubstantiated claims, incomplete steps, or reliance on conjectures \cites{di2013proof, norvig2010beal, dahmen2013perfect, de2016solutions}.  Many partial proofs are published in which limited domains, conditional upper bounds of the exponents, specific configuration of bases or exponents, or additional, relaxed, or modified constraints are applied in which the conjecture holds \cites{beukers1998, beukers2020generalized, siksek2012partial, poonen1998some, merel1997winding, bennett2006equation, anni2016modular, billerey2018some, kraus1998equation, poonen2007twists, siksek2014generalised, miyazaki2015upper}.   

Some researchers demonstrate or prove limited coprimality of the exponents \cite{ghosh2011proof}, properties of perfect powers and relationships to Pillai's conjecture \cite{waldschmidt2009perfect}, impossibility of solutions for specific bases \cite{mihailescu2004primary}, influence of the parity of the exponents \cite{joseph2018another}, characterizations of related Diophantine equations \cite{nathanson2016diophantine}, relationship between the smallest base and the common factor \cite{townsend2010search}, and countless other insights.  


To formally establish a rigorous and complete proof, we need to consider two complimentary conditions: 1) when gcd$(A,B,C)=1$ there is no integer solution to \BealsEq, and 2) if there is an integer solution, then gcd$(A,B,C)>1$.  The approach we take is linked to the properties of slopes.  An integer solution that satisfies the conjecture also marks a point $(A,B,C)$ that subtends a line through the origin on a 3 dimensional Cartesian graph.  Being integers, this is a lattice point and thus the line has a rational slope in all 3 planes.  Among other properties, it will be shown that if gcd$(A,B,C)=1$ with integer exponents, then one or more of the slopes of the subtended line is irrational and cannot pass through any non-trivial lattice points.  Conversely it will be shown that if there exists a solution that satisfies the conjecture, the subtended line must pass through a non-trivial lattice point and must have rational slopes.

\section{Details of the Proof}

To establish the proof requires we first identify, substantiate, and then prove several preliminary properties:
\begin{itemize}
\item \textbf{Slopes of the Terms}:  determine slopes of lines subtended by the origin
              and the lattice point $(A,B,C)$ that satisfy the terms of the conjecture
              (\cref{Thm:2.1_Irrational_Slope_No_Lattice} on
              pages \pageref{Section:Slopes_Start} to \pageref{Section:Slopes_End}).
\item \textbf{Coprimality of the Bases}:  determine implications of 3-way coprimality on
              pairwise coprimality and implications of pairwise coprimality on 3-way
              coprimality
              (\cref{Thm:2.2_Coprime,Thm:2.3_Coprime,Thm:2.4_Coprime}
              on pages \pageref{Section:Comprimality_Start} to
              \pageref{Section:Comprimality_End}).
\item \textbf{Restrictions of the Exponents}:  determine limits of the exponents related
              to coprimality of the bases and bounds of the conjecture
                          (\cref{Thm:2.5_X_cannot_be_mult_of_Z} on pages \pageref{Section:Exponents_Start}
                          to \pageref{Section:Exponents_End}).
\item \textbf{Reparameterization of the Terms}:  determine equivalent functional forms
              of the terms and associated properties as related to coprimality of the terms
              (\cref{Thm:2.6_Initial_Expansion_of_Differences,Thm:2.7_Indeterminate_Limit,Thm:2.8_Functional_Form,Thm:2.9_Real_Alpha_Beta,Thm:2.10_No_Solution_Alpha_Beta_Irrational,Thm:2.11_Coprime_Alpha_Beta_Irrational,Thm:2.12_Rational_Alpha_Beta_Rational_Then_Not_Coprime,Thm:2.13_Coprime_Any_Alpha_Beta_Irrational_Indeterminate} on
              pages \pageref{Section:Reparameterize_Start} to
               \pageref{Section:Reparameterize_End}).
\item \textbf{Impossibility of the Terms}:  determine the relationship between
              term coprimality and slope irrationality, and between
              slope irrationality and solution impossibility
              (\cref{Thm:2.14_Main_Proof_Coprime_No_Solutions} on
              pages \pageref{Section:Impossibility_Start} to \pageref{Section:Impossibility_End}).
\item \textbf{Requirement for Possibility of the Terms}:  determine characteristics of
              gcd$(A,B,C)$ required for there to exist a solution given the
              properties of slopes and coprimality
              (\cref{Thm:2.15_Main_Proof_Solutions_Then_Not_Coprime} on pages
              \pageref{Section:Possibility_Start} to \pageref{Section:Possibility_End}).
\end{itemize}

Before articulating each of the underlying formal proofs, we establish two specific definitions to ensure consistency of interpretation.
\begin{enumerate}
\item \textbf{Reduced Form}:  We define the bases of $A^X$, $B^Y$, and $C^Z$ to be in reduced
                          form, meaning that rather than let the bases be perfect powers, we
                          define  exponents $X$, $Y$, and $Z$ such that the corresponding bases
                          are not perfect powers.  For example, $8^5$ can be reduced to
                          $2^{15}$ and thus the base and exponent would be 2 and 15,
                          respectively, not 8 and 5, respectively.  Hence throughout this
                          document, we assume all bases are reduced accordingly given that the
                          reduced and non-reduced forms of bases raised to their corresponding
                          exponents are equal.
\item $\bm{C^Z=f(A,B,X,Y)}$:  Without loss of generality, when establishing the impossibility
                          of integer solutions, unless stated otherwise, we assume to start
                          with integer values for $A$, $B$, $X$, and $Y$ and then determine
                          the impossibility of integers for $C$ or $Z$.  Given the commutative
                          property of the equation, we hereafter base the determination of
                          integrality of $C$ and $Z$ as a function of definite integrality of
                          $A$, $B$, $X$, and $Y$, as doing so for any other combination of
                          variables is a trivial generalization.
\end{enumerate}

The following hereafter establishes the above objectives.

\Line
\subsection{Slopes of the Terms}
\label{Section:Slopes_Start}
Although the exponents in \BealsEq\, suggest a volumetric relationship between cubes and hypercubes, and given that the exponents cannot all be the same as it would violate Fermat's last theorem \cites{wiles1995modular, taylor1995ring, shanks2001solved}, the expectation of apparent geometric interpretability is low.  However every set of values that satisfy the conjecture correspond to a point on a Cartesian grid and subtend a line segment with the origin, which in turn means properties of the slopes of these line segments are directly related to the conjecture.  Properties of these slopes form a crucial basis for the subsequent main proofs.

\begin{figure}
\large{$C^Z\, vs\, B^Y\, vs\, A^X$}\\
\includegraphics[width=.5\textwidth]{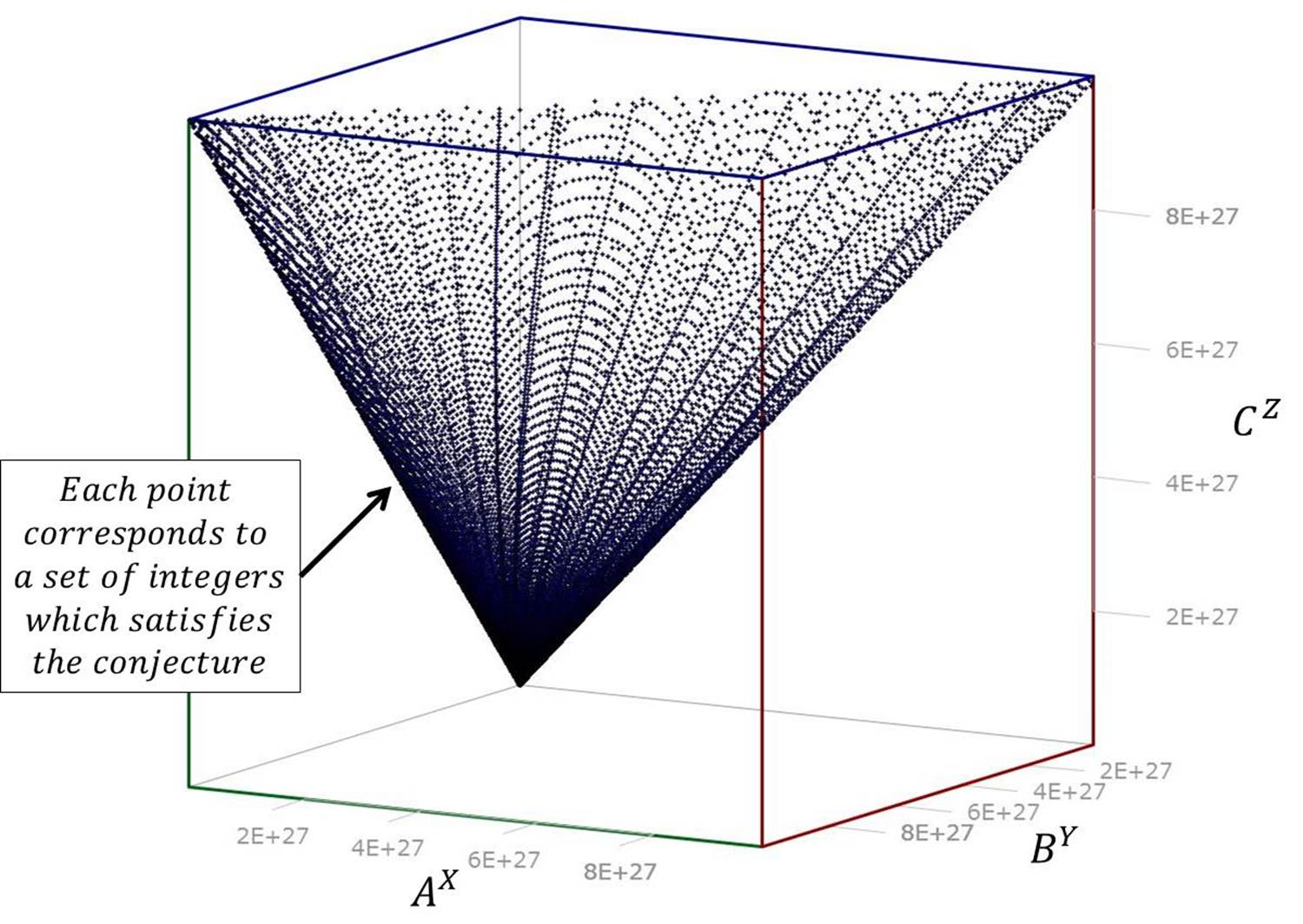}
\caption{Plot of $(A^X,B^Y,C^Z)$ given positive $A$, $B$, $C$, $X$, $Y$, and $Z$, where \BealsEq, $A^X+B^Y\leq10^{28}$, $X\geq4$, $Y,Z\geq3$.}
\label{Fig:CZBYAZScatter}
\end{figure}

In the Cartesian plot of $A^X\times B^Y\times C^Z$, each point $(A^X, B^Y, C^Z)$ corresponds to a specific integer solution that satisfies the conjecture found from exhaustive search within a limited domain.  See \cref{Fig:CZBYAZScatter}.  There exists a unique line segment between each point and the origin.  The line segment subtends a line segment in each of the three planes, and a set of corresponding angles in those planes made with the axes.  See \cref{Fig:3DScatter,Fig:ScatterPlotWithAngles}.

\begin{figure}
\includegraphics[width=.33\textwidth]{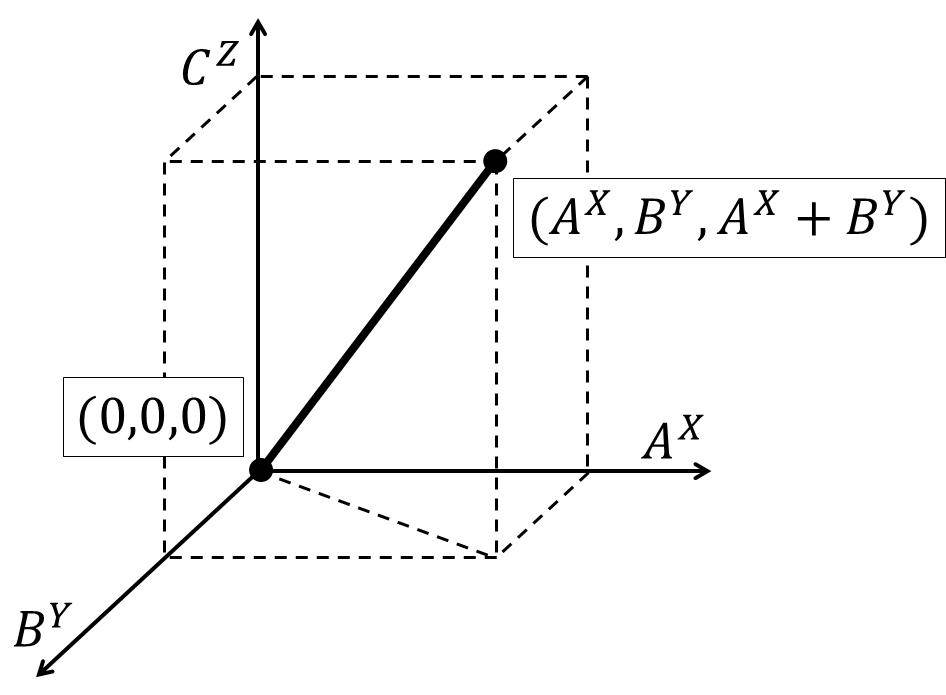}
\caption{Line segment connecting the origin and point $(A^X, B^Y, C^Z)$ where \BealsEq\, satisfying the conjecture from \cref{Fig:CZBYAZScatter}.}
\label{Fig:3DScatter}
\end{figure}

\begin{figure}
\includegraphics[width=.66\textwidth]{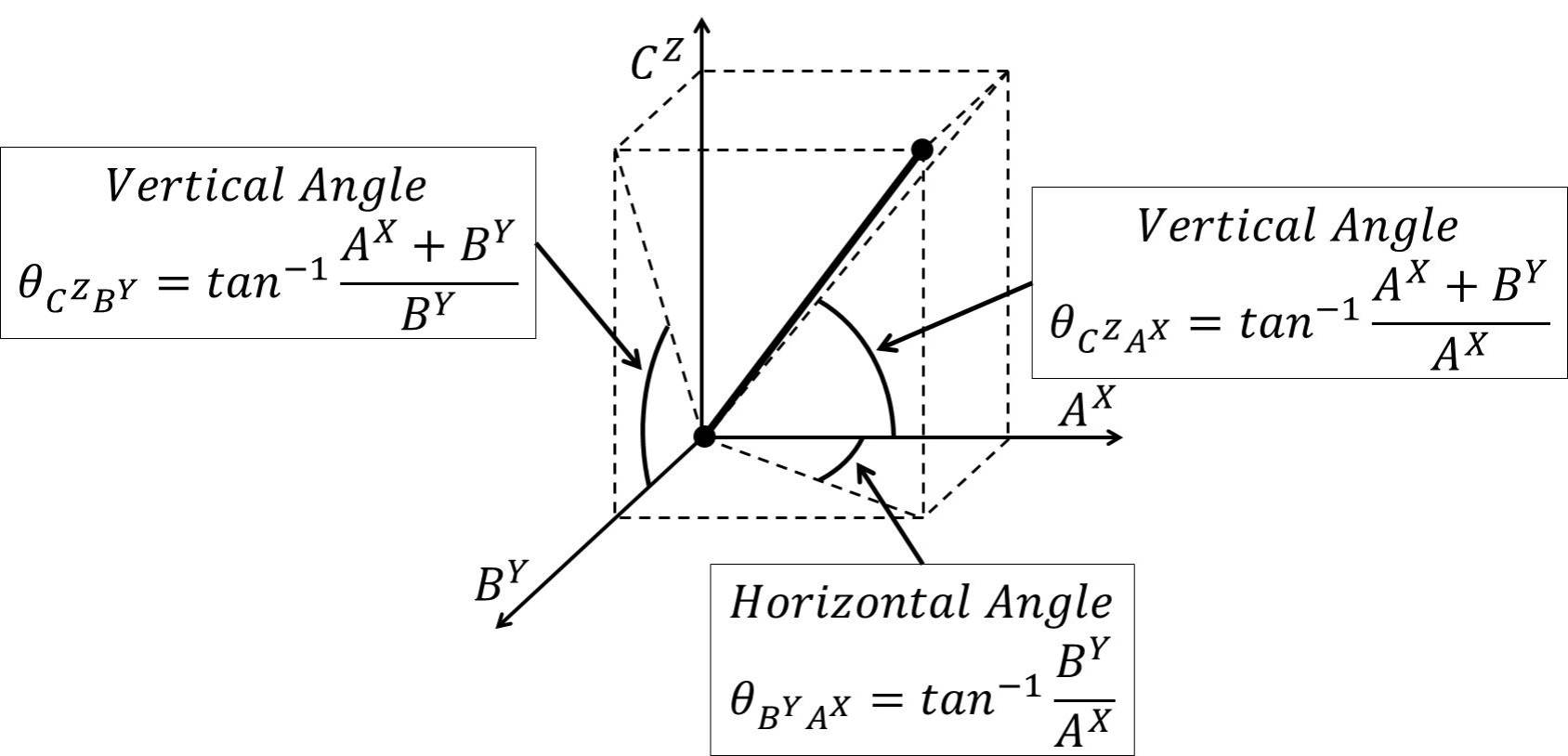}
\caption{Angles between the axes and the line segment subtended by the origin and point $(A^X, B^Y, C^Z)$ from \cref{Fig:CZBYAZScatter}.}
\label{Fig:ScatterPlotWithAngles}
\end{figure}

\bigskip

\begin{theorem}
\label{Thm:2.1_Irrational_Slope_No_Lattice}
If the slope $m$ of line $y=mx$ is irrational, then the line does not go through any non-trivial lattice points.
\end{theorem}

\begin{proof}
Suppose there exists a line $y=mx$, with irrational slope $m$ that goes through lattice point $(X,Y)$.  Then the slope can be calculated from the two known lattice points through which this line traverses, $(0,0)$ and $(X,Y)$.  Hence the slope is $\displaystyle{m=\frac{Y-0}{X-0}}$.  However, $\displaystyle{m=\frac{Y}{X}}$ is the ratio of integers, thus contradicting $m$ being irrational, hence a line with an irrational slope passing through the origin cannot pass through a non-trivial lattice point.
\end{proof}

Since values that satisfy the conjecture are integer, by definition they correspond to a lattice point and thus the line segment between that lattice point and the origin will have a rational slope per \cref{Thm:2.1_Irrational_Slope_No_Lattice}.  We next establish the properties of term coprimality and thereafter the relationship coprimality has on the slope irrationality.  Thereafter we establish several other preliminary proofs relating to reparameterizations and non-standard binomial expansions before returning back to the connection between term coprimality, slope irrationality, and the conjecture proof.
\label{Section:Slopes_End}

\Line
\subsection{Coprimality of the Bases}
\label{Section:Comprimality_Start}
According to the conjecture, solutions only exist when gcd$(A,B,C)>1$.  Hence testing for impossibility when gcd$(A,B,C)=1$ requires we establish the relationship between 3-way coprimality and the more stringent pairwise coprimality.

\bigskip

\begin{theorem}
\label{Thm:2.2_Coprime}
\Conjecture, if gcd$(A,B)=1$, then gcd$(A^X,C^Z )=$ gcd$(B^Y,C^Z )=1$.\\
\{If $A$ and $B$ are coprime, then $C^Z$ is pairwise coprime to $A^X$ and $B^Y$.\}
\end{theorem}

\begin{proof}
Suppose $A$ and $B$ are coprime.  Then $A^X$ and $B^Y$ are coprime, and we can define these terms based on their respective prime factors, namely
\begin{subequations}
\begin{gather}
A^X = \PList \label{eqn:2a} \\
B^Y = \QList \label{eqn:2b}
\end{gather}
\end{subequations}
where $p_i$, $q_j$ are prime, and $p_i\neq q_j$, for all $i,j$.  Based on \cref{eqn:1}, we can express $C^Z$ as
\begin{equation}
C^Z=\PList + \QList \label{eqn:3}
\end{equation}

We now take any prime factor of $A^X$ or $B^Y$ and divide both sides of \cref{eqn:3} by that prime factor.  Without loss of generalization, suppose we choose $p_i$.  Thus
\begin{subequations}
\begin{align}
\frac{C^Z}{p_i} &= \frac{\PList + \QList}{p_i}             \label{eqn:4a} \\
\frac{C^Z}{p_i} &= \frac{\PList}{p_i} + \frac{\QList}{p_i} \label{eqn:4b}
\end{align}
\end{subequations}
The term $\displaystyle\frac{\PList}{p_i}$  is an integer since by definition $p_i |\PList$.  However the term  $\displaystyle\frac{\QList}{p_i}$ cannot be simplified since $p_i \nmid \QList$ and thus $p_i \nmid (\PList+\QList)$.  Hence by extension $p_i \nmid C^Z$ and $A^X$ must thus be coprime to $C^Z$.  By applying the same logic with $q_j$, then $B^Y$ must also be coprime to $C^Z$.  Therefore if $A$ and $B$ are coprime, then $C^Z$ must be pairwise coprime to both $A^X$ and $B^Y$.
\end{proof}

\bigskip

\begin{theorem}
\label{Thm:2.3_Coprime}
\Conjecture, if gcd$(A,C)=1$ or gcd$(B,C)=1$ then gcd$(A^X,C^Z)=$ gcd$(B^Y,C^Z)=1$.\\
\{If $A$ or $B$ is coprime to $C$, then $C^Z$ is pairwise coprime to $A^X$ and $B^Y$.\}
\end{theorem}

\begin{proof}
Without loss of generalization, suppose $A$ and $C$ are coprime.  Thus $A^X$ and $C^Z$ are coprime.  We can define $C^Z$ based on its prime factors, namely
\begin{equation}
C^Z=\RList \label{eqn:5}
\end{equation}
where $r_k$ are primes. Based on \cref{eqn:2a,eqn:2a,eqn:5}, we can define $B^Y$ based on the difference between $C^Z$ and $A^X$, namely
\begin{equation}
B^Y=\RList - \PList \label{eqn:6}
\end{equation}
We now take any prime factor of $C^Z$ and divide both sides of \cref{eqn:6} by that prime factor.  Without loss of generalization, suppose we choose $r_k$.  Thus
\begin{subequations}
\begin{align}
\frac{B^Y}{r_k} &= \frac{\RList - \PList}{r_k}             \label{eqn:7a} \\
\frac{B^Y}{r_k} &= \frac{\RList}{r_k} - \frac{\PList}{r_k} \label{eqn:7b}
\end{align}
\end{subequations}
The term $\displaystyle\frac{\RList}{r_k}$ is an integer since by definition $r_k |\RList$.  However the term  $\displaystyle\frac{\PList}{r_k}$ cannot be simplified since $r_k \nmid \PList$ and thus  $r_k \nmid (\RList - \PList)$.  Hence by extension $r_k \nmid B^Y$ and $C^Z$ must thus be coprime to $B^Y$.  By applying the same logic with $p_i$, then $C^Z$ must also be coprime to $A^X$.  Therefore if either $A$ or $B$ is coprime to $C$, then $C^Z$ must be pairwise coprime to both $A^X$ and $B^Y$.
\end{proof}

\bigskip

\begin{theorem}
\label{Thm:2.4_Coprime}
\Conjecture, if gcd$(A,B,C) = 1$ then gcd $(A^X,B^Y)=$ gcd $(A^X,C^Z)=$ gcd $(B^Y,C^Z)=1$\\
\{If $A$, $B$ and $C$ are 3-way coprime, then they are all pairwise coprime.\}
\end{theorem}

\begin{proof} We consider two scenarios when gcd$(A,B,C)=1$, namely: gcd$(A,B)>1$ and gcd$(A,C)>1$ (the later of which generalizes to gcd$(B,C)>1$).

\bigskip

\textbf{Scenario 1 of 2:}  Suppose gcd$(A,B,C)=1$ while gcd$(A,B)>1$.  Therefore $A$ and $B$ have a common factor.  Thus we can express $A^X$ and $B^Y$ relative to their common factor, namely
\begin{subequations}
\begin{gather}
A^X = k\cdot\PList \label{eqn:8a} \\
B^Y = k\cdot\QList \label{eqn:8b}
\end{gather}
\end{subequations}
where integer $k$ is the common factor, and $p_i$, $q_j$ are prime, and $p_i\neq q_j$, for all $i,j$.  Based on \cref{eqn:1}, we can express $C^Z$ as
\begin{subequations}
\begin{gather}
C^Z=k\cdot\PList + k\cdot\QList \label{eqn:9a} \\
C^Z=k(\PList + \QList)          \label{eqn:9b}
\end{gather}
\end{subequations}
Per \cref{eqn:9b}, $k$ is a factor of $C^Z$, just as it is a factor of $A^X$ and $B^Y$, thus gcd$(A,B,C)\neq1$, hence a contradiction.  Thus when gcd$(A,B,C)=1$ we know gcd$(A,B)\ngtr1$ and thus $k$ must be 1.

\bigskip

\textbf{Scenario 2 of 2:}  Suppose gcd$(A,B,C)=1$ while gcd$(A,C)>1$.  Therefore $A$ and $C$ have a common factor.  Thus we can express $A^X$ and $C^Z$ relative to their common factor, namely
\begin{subequations}
\begin{gather}
A^X = k\cdot\PList \label{eqn:10a} \\
C^Z = k\cdot\RList \label{eqn:10b}
\end{gather}
\end{subequations}
where integer $k$ is the common factor, and $p_i$, $r_k$ are prime, and $p_i\neq r_k$, for all $i,k$.  Based on \cref{eqn:1}, we can express $B^Y$ as
\begin{subequations}
\begin{gather}
B^Y=k\cdot\RList - k\cdot\PList \label{eqn:11a} \\
B^Y=k(\RList - \PList)          \label{eqn:11b}
\end{gather}
\end{subequations}
Per \cref{eqn:11b}, $k$ is a factor of $B^Y$, just as it is a factor of $A^X$ and $C^Z$, thus gcd$(A,B,C)\neq1$, hence a contradiction.  Thus when gcd$(A,B,C)=1$ we know gcd$(A,C)\ngtr1$ and thus $k$ must be 1.  By extension and generalization, when gcd$(A,B,C)=1$ we know gcd$(B,C)\ngtr1$.
\end{proof}

Based on \cref{Thm:2.2_Coprime,Thm:2.3_Coprime,Thm:2.4_Coprime} if any pair of terms $A$, $B$, and $C$ have no common factor, then all pairs of terms are coprime.  Hence either all three terms share a common factor or they are all pairwise coprime.  We thus formally conclude that if gcd$(A, B, C)=1$, then gcd$(A^X,B^Y)=$ gcd$(A^X,C^Z)=$ gcd$(B^Y,C^Z)=1$, and if gcd$(A,B)=1$ or gcd$(A,C)=1$ or gcd$(B,C)=1$, then gcd$(A,B,C)=1$.
\label{Section:Comprimality_End}

\bigskip

\Line
\subsection{Restrictions of the Exponents}
\label{Section:Exponents_Start}
Trivial restrictions of the exponents are defined by the conjecture, namely integer and greater than 2.  However, other restrictions apply such as, per Fermat's last theorem, the exponents cannot be equal while greater than 2.  More subtle restrictions also apply which will be required for the main proofs.

\begin{theorem}
\label{Thm:2.5_X_cannot_be_mult_of_Z}
\Conjecture, exponents $X$ and $Y$ cannot simultaneously be integer multiples of exponent $Z$.
\end{theorem}

\begin{proof}
Suppose $X$ and $Y$ are simultaneously each integer multiples of $Z$.  Thus $X=jZ$ and $Y=kZ$ for positive integers $j$ and $k$.  Therefore we can restate \cref{eqn:1} as
\begin{equation}
(A^j)^Z + (B^k)^Z= C^Z  \label{eqn:12}
\end{equation}
Per \cref{eqn:12}, we have 3 terms which are each raised to exponent $Z\ge3$.  According to Fermat's last theorem \cites{wiles1995modular, taylor1995ring, shanks2001solved}, no integer solution exists when the terms share a common exponent greater than 2.  Therefore  $X$ and $Y$ cannot simultaneously be integer multiples of $Z$.
\end{proof}

A trivially equivalent variation of \cref{Thm:2.5_X_cannot_be_mult_of_Z} is that $Z$ cannot simultaneously be a unit fraction of $X$ and $Y$.  Given \cref{Thm:2.5_X_cannot_be_mult_of_Z}, there are only two possibilities:
\begin{enumerate}
\item Neither $X$ or $Y$ are multiples of $Z$.
\item Only one of either $X$ or $Y$ is a multiple of $Z$.
\end{enumerate}
As such, given that at least one of the two exponents $X$ and $Y$ cannot be a multiple of $Z$, of the terms $A^X$ and $B^Y$, we therefore can arbitrarily choose $A^X$ to be the term whose exponent is not an integer multiple of exponent $Z$.  Hence the following definition is used hereafter:

\begin{definition}
\label{Dfn:2.1_X_cannot_be_mult_of_Z}
$X$ is not an integer multiple of $Z$.
\end{definition}

Since per \cref{Thm:2.5_X_cannot_be_mult_of_Z} at most only one of $X$ or $Y$ can be a multiple of $Z$ and given one can arbitrarily swap $A^X$ and $B^Y$, the arbitrary fixing hereafter of $A^X$ to be the term for which its exponent is not a multiple of $Z$ does not interfere with any of the characteristics or implications of the solution.  Hence we hereafter define $A^X$ and $B^Y$ such that \cref{Dfn:2.1_X_cannot_be_mult_of_Z} is maintained.

\label{Section:Exponents_End}

\bigskip

\Line
\subsection{Reparameterization of the Terms}
\label{Section:Reparameterize_Start}
In exploring ways to leverage the binomial expansion and other equivalences, some researchers \cites{beauchamp2018,beauchamp2019,edwards2005platonic} explored reparameterizing one or more of the terms of \BealsEq\, so as to compare different sets of expansions.  We broaden this idea to establish various irrationality conditions as related to coprimality of the terms, establish properties of the non-unique characteristics of key terms in the expansions, and showcase an exhaustive view to be leveraged when validating the conjecture.

The binomial expansion applied to the difference of perfect powers with different exponents is critical to mathematical research in general and to several proofs specifically later in this document.  One feature of the binomial expansion in our application is the circumstance under which the upper limit of the sum is indeterminate \cites{beauchamp2018,beauchamp2019} to be introduced in the following two theorems.

\bigskip

\begin{theorem}
\label{Thm:2.6_Initial_Expansion_of_Differences}
If $p,q\neq 0$ and $v,w$ are real, then $\displaystyle{p^v-q^w=(p+q)(p^{v-1}-q^{w-1}) - pq(p^{v-2} - q^{w-2})}$.\\
\{Expanding the difference of two powers\}
\end{theorem}

\begin{proof}
Given non-zero $p$ and $q$, and real $v$ and $w$, suppose we can expand the difference $p^v-q^w$ as
\begin{equation}
p^v-q^w=(p+q)(p^{v-1}-q^{w-1}) - pq(p^{v-2} - q^{w-2}) \label{eqn:13}
\end{equation}
Distributing $(p+q)$ on the right side of \cref{eqn:13} into $(p^{v-1}-q^{w-1})$ gives us $\left[p^v-pq^{w-1} + p^{v-1}q-q^w\right]$ and distributing $-pq$ into $(p^{v-2} - q^{w-2})$ gives us $\left[-p^{v-1}q+pq^{w-1}\right]$.  Thus simplifying \cref{eqn:13} gives us
\begin{subequations}
\begin{align}
p^v-q^w &= \left[p^v-pq^{w-1} + p^{v-1}q-q^w\right] +\left[-p^{v-1}q+pq^{w-1}\right] \label{eqn:14a}       \\
p^v-q^w &= p^v+\left[pq^{w-1}-pq^{w-1}\right] + \left[p^{v-1}q -p^{v-1}q\right]  -q^w\label{eqn:14b} \\
p^v-q^w &= p^v  -q^w\label{eqn:14c}
\end{align}
\end{subequations}
\end{proof}
Thus the difference of powers can indeed be expanded per the above functional form accordingly.  We also observe \cref{eqn:13} can be expressed in more compact notation, namely
\begin{equation}
p^v-q^w=\sum \limits_{i=0}^1 (p+q)^{1-i}(-pq)^i(p^{v-1-i}-q^{w-1-i}) \label{eqn:15}
\end{equation}
We further observe in \cref{eqn:13} of \cref{Thm:2.6_Initial_Expansion_of_Differences} that this expansion of the difference of two powers yields two other terms which are  themselves differences of powers, namely $(p^{v-1}-q^{w-1})$ and $(p^{v-2} - q^{w-2})$.  Each of these differences could likewise be expanded with the same functional form of \cref{Thm:2.6_Initial_Expansion_of_Differences}.  Recursively expanding the resulting terms of differences of powers leads to a more general form of \cref{eqn:15}.

\bigskip

\begin{theorem}
\label{Thm:2.7_Indeterminate_Limit}
If $p,q\neq 0$ and integer $n\geq0$, then $\displaystyle{p^v-q^w =\sum \limits_{i=0}^n \binom{n}{i} (p+q)^{n-i}(-pq)^i(p^{v-n-i}-q^{w-n-i})}$.\\
\{General form of the expansion of the difference of two powers\}
\end{theorem}

\begin{proof}
Suppose $p,q\neq 0$ and integer $n\geq0$, and suppose
\begin{equation}
p^v-q^w =\sum \limits_{i=0}^n \binom{n}{i} (p+q)^{n-i}(-pq)^i(p^{v-n-i}-q^{w-n-i})
\label{eqn:16}
\end{equation}
Consider $n=0$.  The right side of \cref{eqn:16} reduces to $p^v-q^w$, thus \cref{eqn:16} holds when $n=0$.  Consider $n=1$.  The right side of \cref{eqn:16} becomes
\begin{subequations}
\begin{align}
\begin{split}
p^v-q^w &=\binom{1}{0} (p+q)^{1-0}(-pq)^0(p^{v-1-0}-q^{w-1-0})
        +\binom{1}{1} (p+q)^{1-1}(-pq)^1(p^{v-1-1}-q^{w-1-1})
                 \label{eqn:17a}
\end{split}\\ 
\begin{split}
p^v-q^w &=  \biggl[ (p+q)(p^{v-1}-q^{w-1}) \biggr] +  \biggl[(-pq)(p^{v-2}-q^{w-2})\biggr] \\
        &= p^v-q^w
                 \label{eqn:17b}
\end{split}
\end{align}
\end{subequations}
The right side of \cref{eqn:17b} also reduces to $p^v-q^w$.  Hence \cref{eqn:16} holds for $n=0$ and $n=1$.


In generalizing, enumerating the terms of \cref{eqn:16} gives us
\begin{equation}
\begin{split}
p^v-q^w &=
 \binom{n}{0}   (p+q)^n           (p^{v-n}   -q^{w-n  }) \\
&+\binom{n}{1}  (p+q)^{n-1}(-pq)  (p^{v-n-1} -q^{w-n-1}) \\
&+\binom{n}{2}  (p+q)^{n-2}(-pq)^2(p^{v-n-2} -q^{w-n-2}) \\
&+ \ \cdots \\
&+\binom{n}{n-2}(p+q)^2(-pq)^{n-2}(p^{v-2n+2}-q^{w-2n+2}) \\
&+\binom{n}{n-1}(p+q)  (-pq)^{n-1}(p^{v-2n+1}-q^{w-2n+1}) \\
&+\binom{n}{n}         (-pq)^n    (p^{v-2n} - q^{w-2n})   \\
\end{split}
\label{eqn:18}
\end{equation}
Expanding each of the $n+1$ differences of powers $(p^{v-n-i}-q^{w-n-i})$ of \cref{eqn:18} per \cref{Thm:2.6_Initial_Expansion_of_Differences} gives us

\begin{equation}
\begin{split}
p^v-q^w &=
 \binom{n}{0}   (p+q)^n           \left[(p+q)(p^{v-n-1}-q^{w-n-1}) - pq(p^{v-n-2} - q^{w-n-2}) \right] \\
&+\binom{n}{1}  (p+q)^{n-1}(-pq)  [(p+q)(p^{v-n-2}-q^{w-n-2}) - pq(p^{v-n-3} - q^{w-n-3}) ] \\
&+\binom{n}{2}  (p+q)^{n-2}(-pq)^2[(p+q)(p^{v-n-3}-q^{w-n-3}) - pq(p^{v-n-4} - q^{w-n-4}) ] \\
&+ \ \cdots \\
&+\binom{n}{n-2}(p+q)^2(-pq)^{n-2}[(p+q)(p^{v-2n+1}-q^{w-2n+1}) - pq(p^{v-2n} - q^{w-2n}) ] \\
&+\binom{n}{n-1}(p+q)  (-pq)^{n-1}[(p+q)(p^{v-2n}-q^{w-2n}) - pq(p^{v-2n-1} - q^{w-2n-1}) ] \\
&+\binom{n}{n}         (-pq)^n    [(p+q)(p^{v-2n-1}-q^{w-2n-1}) - pq(p^{v-2n-2} - q^{w-2n-2}) ]   \\
\end{split}
\label{eqn:19}
\end{equation}
Distributing each of the $\displaystyle{\binom{n}{i}(p+q)^{n-i}}(-pq)^i$ terms of \cref{eqn:19} into the corresponding bracketed terms then gives us
\begin{equation}
\begin{split}
p^v-q^w &=
 \binom{n}{0}   (p+q)^{n+1} (p^{v-n-1}-q^{w-n-1}) + \binom{n}{0}(p+q)^n(-pq)(p^{v-n-2} - q^{w-n-2}) \\
&+\binom{n}{1}  (p+q)^n(-pq)(p^{v-n-2}-q^{w-n-2}) + \binom{n}{1}(p+q)^{n-1}(-pq)^2(p^{v-n-3} - q^{w-n-3}) \\
&+\binom{n}{2}  (p+q)^{n-1}(-pq)^2(p^{v-n-3}-q^{w-n-3}) + \binom{n}{2}(p+q)^{n-2} (-pq)^3(p^{v-n-4} - q^{w-n-4}) \\
&+ \ \cdots \\
&+\binom{n}{\!n-2\!}(p+q)^3(-pq)^{n-2}(p^{v-2n+1}\!-\!q^{w-2n+1}) + \binom{n}{\!n-2\!}(p+q)^2 (-pq)^{n-1}(p^{v-2n} \!-\! q^{w-2n}) \\
&+\binom{n}{\!n-1\!}(p+q)^2(-pq)^{n-1}(p^{v-2n}\!-\!q^{w-2n}) + \binom{n}{\!n-1\!}(p+q) (-pq)^n(p^{v-2n-1} \!-\! q^{w-2n-1}) \\
&+\binom{n}{n}  (p+q)  (-pq)^n    (p^{v-2n-1}-q^{w-2n-1}) + \binom{n}{n} (-pq)^{n+1}(p^{v-2n-2} - q^{w-2n-2})  \\
\end{split}
\label{eqn:20}
\end{equation}
which can be simplified to
\begin{equation}
\begin{split}
p^v-q^w &=
 \binom{n}{0}   (p+q)^{n+1} (p^{v-n-1}-q^{w-n-1}) \\
&+\left[\binom{n}{1}+\binom{n}{0}\right](p+q)^n(-pq)(p^{v-n-2}-q^{w-n-2})  \\
&+\left[\binom{n}{2}+\binom{n}{1}\right](p+q)^{n-1}(-pq)^2(p^{v-n-3}-q^{w-n-3})\\
&+\left[\binom{n}{3}+\binom{n}{2}\right](p+q)^{n-2}(-pq)^3(p^{v-n-4}-q^{w-n-4})\\
&+ \ \cdots \\
&+\left[\binom{n}{n-2}+\binom{n}{n-3}\right](p+q)^3(-pq)^{n-2}(p^{v-2n+1}-q^{w-2n+1}) \\
&+\left[\binom{n}{n-1}+\binom{n}{n-2}\right](p+q)^2(-pq)^{n-1}(p^{v-2n}-q^{w-2n})  \\
&+\left[\binom{n}{n}+\binom{n}{n-1}\right]  (p+q)(-pq)^n     (p^{v-2n-1}-q^{w-2n-1})\\
&+ \binom{n}{n} (-pq){n+1}(p^{v-2n-2} - q^{w-2n-2})  \\
\end{split}
\label{eqn:21}
\end{equation}
Pascal's identity states that $\displaystyle{\binom{m+1}{k} = \binom{m}{k}+\binom{m}{k-1}}$ for integer $k\geq1$ and integer $m\geq0$ \cite{macmillan2011proofs,fjelstad1991extending}.  Given this identity, each sum of the pairs of binomial coefficients in the brackets of \cref{eqn:21} simplifies to:
\begin{subequations}
\begin{gather}
\begin{split}
p^v-q^w &=
 \binom{n}{0}      (p+q)^{n+1}       (p^{v-n-1} -q^{w-n-1}) \\
&+\binom{n+1}{1}   (p+q)^n    (-pq)  (p^{v-n-2} -q^{w-n-2}) \\
&+\binom{n+1}{2}   (p+q)^{n-1}(-pq)^2(p^{v-n-3} -q^{w-n-3}) \\
&+\binom{n+1}{3}   (p+q)^{n-2}(-pq)^3(p^{v-n-4} -q^{w-n-4}) \\
&+ \ \cdots \\
&+\binom{n+1}{n-2} (p+q)^3(-pq)^{n-2}(p^{v-2n+1}-q^{w-2n+1}) \\
&+\binom{n+1}{n-1} (p+q)^2(-pq)^{n-1}(p^{v-2n}  -q^{w-2n})   \\
&+\binom{n+1}{n}   (p+q)  (-pq)^n    (p^{v-2n-1}-q^{w-2n-1}) \\
&+ \binom{n}{n}           (-pq)^{n+1}(p^{v-2n-2}-q^{w-2n-2}) \\
\end{split} \label{eqn:22a} \\
p^v-q^w =\sum \limits_{i=0}^{n+1} \binom{n+1}{i} (p+q)^{n-i+1}(-pq)^i(p^{v-n-i-1}-q^{w-n-i-1})
\label{eqn:22b}
\end{gather}
\end{subequations}
The right side of \cref{eqn:16} and \cref{eqn:22b} both equal $\displaystyle{p^v-q^w}$, thus the right side of these two equations are equal.  Hence
\begin{equation}
\sum \limits_{i=0}^n \binom{n}{i} (p+q)^{n-i}(-pq)^i(p^{v-n-i}-q^{w-n-i}) = \sum \limits_{i=0}^{n+1} \binom{n+1}{i} (p+q)^{n-i+1}(-pq)^i(p^{v-n-i-1}-q^{w-n-i-1})
\label{eqn:23}
\end{equation}
Therefore by induction, since $\displaystyle{p^v-q^w =\sum \limits_{i=0}^n \binom{n}{i} (p+q)^{n-i}(-pq)^i(p^{v-n-i}-q^{w-n-i})}$ for $n=0$ and $n=1$ and per \cref{eqn:22b,eqn:23} this relation holds when $n$ is replaced with $n+1$.  Hence this relation holds for all integers $n\geq0$.
\end{proof}

We observe an important property, per \cref{eqn:16,eqn:22b,eqn:23}, that $n$ is indeterminate since $\displaystyle{p^v-q^w =\sum \limits_{i=0}^n \binom{n}{i} (p+q)^{n-i}(-pq)^i(p^{v-n-i}-q^{w-n-i})}$ holds for every non-negative integer value of $n$.  Hence any non-negative integer value of $n$ can be selected and the resulting expansion still applies, leading to different expansions that sum to identical outcomes.

\bigskip

Other preliminary properties required for the proof of the conjecture include the fact that each of the perfect powers $A^X$, $B^Y$ and $C^Z$ can be expressed as a linear combination of an additive and multiplicative form of the bases of the other two terms.  This property will reveal a variety of equivalences across various domains.  We need to first establish a few basic principles.

\bigskip

\begin{theorem}
\label{Thm:2.8_Functional_Form}
\Conjecture, if there exists an integer solution, then there exists non-zero positive rational $\alpha$ and $\beta$ such that $A^X=[(C+B)\alpha-CB\beta]^X$.
\end{theorem}

\begin{proof}
Given \cref{eqn:1}, by definition $A^X=C^Z-B^Y$.  If there exist integer solutions that satisfy the conjecture then $A$ and $\sqrt[X]{C^Z-B^Y}$ must be integers.  Suppose there exists non-zero rational $\alpha$ and $\beta$ such that $A=(C+B)\alpha-CB\beta$, then these two expressions for $A$ are identical, namely
\begin{equation}
\sqrt[X]{C^Z-B^Y} = (C+B)\alpha-CB\beta \label{eqn:24}
\end{equation}
Solving for $\alpha$ when given any rational $\beta>0$ gives us $\alpha=\displaystyle{\frac{\sqrt[X]{C^Z-B^Y}+CB\beta}{C+B}}$, which is positive since the numerator and denominator are each positive.  Further, if $\beta$ is rational, then $\alpha$ must also be rational since $\displaystyle{\sqrt[X]{C^Z-B^Y}}$, $C$, and $B$ are each integer.

Solving instead for $\beta$ when given any arbitrary sufficiently large positive rational $\alpha$ gives us $\beta=\displaystyle{\frac{\sqrt[X]{C^Z-B^Y}-(C+B)\alpha}{-CB}}$, which is positive given the numerator and denominator are both negative integers.  Further, if $\alpha$ is rational, then $\beta$ must also be rational since $\displaystyle{\sqrt[X]{C^Z-B^Y}}$, $C$, and $B$ are each integer.

Hence there exist non-zero positive rational $\alpha$ and $\beta$ such that $A^X=C^Z-B^Y$ and  $A^X=[(C+B)\alpha-CB\beta]^X$, when the terms of the conjecture are satisfied.
\end{proof}

\bigskip

Without loss of generalization, it can be trivially shown that \cref{Thm:2.8_Functional_Form} establishes an alternate functional form in which $A^X=[(C+B)\alpha-CB\beta]^X$ also applies to $B^Y$ wherein there exists other non-zero positive rational $\alpha$ and $\beta$ such that $B^Y=[(C+A)\alpha-CA\beta]^Y$.

Suppose we arbitrarily let $\displaystyle{\alpha=\sqrt[X]{|C^{Z-X}-B^{Y-X}|}}$ then per \cref{Thm:2.8_Functional_Form}, we can solve for $\beta$ from $C^Z-B^Y=[(C+B)\alpha-CB\beta]^X$ which gives us $\beta=\displaystyle{\frac{\sqrt[X]{C^Z-B^Y}-(C+B)\alpha}{-CB}}$.  Likewise if we arbitrarily let $\displaystyle{\beta=\sqrt[X]{|C^{Z-2X}-B^{Y-2X}|}}$ then per \cref{Thm:2.8_Functional_Form}, we can solve for $\alpha$ from  $C^Z-B^Y=[(C+B)\alpha-CB\beta]^X$ which gives us $\alpha=\displaystyle{\frac{\sqrt[X]{C^Z-B^Y}+CB\beta}{C+B}}$.  In either case, this yields set $\{\alpha,\beta\}$ that satisfies $C^Z-B^Y=[(C+B)\alpha-CB\beta]^X$ which shows these definitions of $\alpha$ and $\beta$ maintain $A^X$ as a perfect power of an integer based on $B$ and $C$, namely $A^X=[(C+B)\alpha-CB\beta]^X$ while satisfying $A^X=C^Z-B^Y$.  Further, based on the indeterminacy of the upper bound in the binomial expansion of the difference of two perfect powers from \cref{Thm:2.7_Indeterminate_Limit}, we can also find values of $\alpha$ and $\beta$ that are explicitly functions of $C$ and $B$.

\bigskip

\begin{theorem}
\label{Thm:2.9_Real_Alpha_Beta}
\Conjecture, values $\displaystyle{\alpha=\sqrt[X]{C^{Z-X}-B^{Y-X}}}$, $\displaystyle{\beta=\sqrt[X]{C^{Z-2X}-B^{Y-2X}}}$, and real $M$ satisfy $C^Z-B^Y=[(C+B)M\alpha-CBM\beta]^X$.
\end{theorem}

\begin{proof}  Given \cref{eqn:1}, we know that integer $A^X=C^Z-B^Y$ can be expanded with the general binomial expansion from \cref{Thm:2.7_Indeterminate_Limit} for the difference of perfect powers, namely
\begin{equation}
A^X=C^Z-B^Y=\sum \limits_{i=0}^n  \binom{n}{1} (C+B)^{n-i}(-CB)^i(C^{Z-n-i}-B^{Y-n-i})
\label{eqn:25}
\end{equation}
Per \cref{Thm:2.7_Indeterminate_Limit}, since upper limit $n$ in \cref{eqn:25} is indeterminate, we can replace $n$ with any value such as $X$ while entirely preserving complete integrity of the terms, hence
\begin{equation}
A^X=C^Z-B^Y=\underbrace{\sum \limits_{i=0}^X \binom{X}{i} (C+B)^{X-i}(-CB)^i}_{Common\,to\,Equation\,(\ref{eqn:27b})\,below}(C^{Z-X-i}-B^{Y-X-i})
\label{eqn:26}
\end{equation}
Furthermore, from \cref{Thm:2.8_Functional_Form} we know that $A=(C+B)\alpha-CB\beta$ for non-zero real $\alpha$ and $\beta$.  Raising $A=(C+B)\alpha-CB\beta$ to $X$ and then expanding gives us
\begin{subequations}
\begin{align}
A^X=\bigl((C+B)\alpha-CB\beta\bigr)^X&=\sum \limits_{i=0}^X  \binom{X}{i}
                      \bigl((C+B)\alpha\bigr)^{X-i}
                      (-CB\beta)^i
\label{eqn:27a} \\
A^X=\bigl((C+B)\alpha-CB\beta\bigr)^X&=\underbrace{ \sum \limits_{i=0}^X \binom{X}{i}
                      (C+B)^{X-i}(-CB)^i}_{Common\,to\,Equation\,(\ref{eqn:26})\,above}
                      \alpha^{X-i}\beta^i
\label{eqn:27b}
\end{align}
\end{subequations}
\cref{eqn:26,eqn:27b} have an identical number of terms, share the identical binomial coefficient, and share $(C+B)^{X-i}(-CB)^i$ for each value of $i$.  See the expansion in \cref{tab:TableSideBySideExpansion}.

\begin{table}[h!]
\begin{center}
\caption{Term-by-term comparison of the binomial expansions of \cref{eqn:26,eqn:27b}.}
\label{tab:TableSideBySideExpansion}
\begin{tabular}{c|c c c}
\textbf{Term} & \textbf{Terms Common to} & \multicolumn{2}{c}{\textbf{\quad Terms Unique to Equations}}  \\
\textbf{\bm{$i$}} & \textbf{\cref{eqn:26,eqn:27b}} & \textbf{(\ref{eqn:26})} & \textbf{(\ref{eqn:27b})} \\
\hline
0      & $\displaystyle{\binom{X}{0}(C+B)^X}$            & $C^{Z-X}-B^{Y-X}$       & $\alpha^{X}$\\
1      & $\displaystyle{\binom{X}{1}(C+B)^{X-1}(-CB)}$   & $C^{Z-X-1}-B^{Y-X-1}$   & $\alpha^{X-1}\beta^1$\\
2      & $\displaystyle{\binom{X}{2}(C+B)^{X-2}(-CB)^2}$ & $C^{Z-X-2}-B^{Y-X-2}$   & $\alpha^{X-2}\beta^2$\\
\vdots & \vdots               & \vdots                    & \vdots\\
$X-1$  & $\displaystyle{\binom{X}{X-1}(C+B)(-CB)^{X-1}}$ & $C^{Z-2X+1}-B^{Y-2X+1}$ & $\alpha^1\beta^{X-1}$\\
$X$    & $\displaystyle{\binom{X}{X}(-CB)^X}$            & $C^{Z-2X}-B^{Y-2X}$     & $\beta^X$\\
\end{tabular}
\end{center}
\end{table}

Per \cref{eqn:26,eqn:27b,Thm:2.8_Functional_Form}, $A^X$ equals both $C^Z-B^Y$ and $[(C+B)\alpha-CB\beta]^X$, thus the finite sums of \cref{eqn:26,eqn:27b} are equal.  Each of the terms in the corresponding expansions have common components and unique components.  Thus there exists a term-wise map between the unique components of the two sets of expansions.

When $i=0$, we observe in \cref{tab:TableSideBySideExpansion} that $\alpha^{X}$ maps to $C^{Z-X}-B^{Y-X}$.  Hence if $\displaystyle{\alpha=\sqrt[X]{C^{Z-X}-B^{Y-X}}}$ or more generically $\displaystyle{\alpha=\sqrt[X]{|C^{Z-X}-B^{Y-X}|}}$, then per \cref{Thm:2.8_Functional_Form}, the corresponding $\beta$ is $\beta=\displaystyle{\frac{\sqrt[X]{C^Z-B^Y}-(C+B)\alpha}{-CB}}$.

When $i=X$, we observe in \cref{tab:TableSideBySideExpansion} that $\beta^X$ maps to $C^{Z-2X}-B^{Y-2X}$.  Hence if $\displaystyle{\beta=\sqrt[X]{C^{Z-2X}-B^{Y-2X}}}$ or more generically $\displaystyle{\beta=\sqrt[X]{|C^{Z-2X}-B^{Y-2X}|}}$, then per \cref{Thm:2.8_Functional_Form}, the corresponding $\alpha$ is $\alpha=\displaystyle{\frac{\sqrt[X]{C^Z-B^Y}+CB\beta}{C+B}}$.

Using the map between $\alpha^{X}$ to $C^{Z-X}-B^{Y-X}$ based on $i=0$ and the map between $\beta^X$ to $C^{Z-2X}-B^{Y-2X}$ based on $i=X$, we can align to the other terms in the expansion of \cref{eqn:26,eqn:27b}.  When $i=1$, we have the terms $C^{Z-X-1}-B^{Y-X-1}$ and $\alpha^{X-1}\beta$ (see \cref{tab:TableSideBySideExpansion}).  Substituting $\displaystyle{\alpha=\sqrt[X]{C^{Z-X}-B^{Y-X}}}$ and $\displaystyle{\beta=\sqrt[X]{C^{Z-2X}-B^{Y-2X}}}$ we have
\begin{subequations}
\begin{gather}
C^{Z-X-1}-B^{Y-X-1}  = \alpha^{X-1}\beta \label{eqn:28a} \\
C^{Z-X-1}-B^{Y-X-1}  = (C^{Z-X}-B^{Y-X})^\frac{X-1}{X}(C^{Z-2X}-B^{Y-2X})^{\frac{1}{X}} \label{eqn:28b}
\end{gather}
\end{subequations}
\cref{eqn:28b} holds in only trivial conditions.  Hence $\alpha$ and $\beta$ cannot simultaneously equal $\displaystyle{\alpha=\sqrt[X]{C^{Z-X}-B^{Y-X}}}$ and $\displaystyle{\beta=\sqrt[X]{C^{Z-2X}-B^{Y-2X}}}$. As such, there are only three possibilities regarding $\alpha$ and $\beta$ that ensure $C^Z-B^Y=[(C+B)\alpha-CB\beta]^X$:
\begin{enumerate}
\item $\alpha$ is arbitrarily defined and thus $\beta=\displaystyle{\frac{\sqrt[X]{C^Z-B^Y}-(C+B)\alpha}{-CB}}$.
\item $\beta$ is arbitrarily defined and thus $\alpha=\displaystyle{\frac{\sqrt[X]{C^Z-B^Y}+CB\beta}{C+B}}$.
\item $\displaystyle{\alpha=\sqrt[X]{C^{Z-X}-B^{Y-X}}}$ and $\displaystyle{\beta=\sqrt[X]{C^{Z-2X}-B^{Y-2X}}}$ where one or both are scaled.
\end{enumerate}
The first two cases, given $\alpha$ or $\beta$ and the other derived therefrom per \cref{Thm:2.8_Functional_Form}, will satisfy $C^Z-B^Y=[(C+B)\alpha-CB\beta]^X$.  In the third case, since $\displaystyle{\alpha=\sqrt[X]{C^{Z-X}-B^{Y-X}}}$ and $\displaystyle{\beta=\sqrt[X]{C^{Z-2X}-B^{Y-2X}}}$ do not simultaneously satisfy $C^Z-B^Y=[(C+B)\alpha-CB\beta]^X$, then scaling $\alpha$ and $\beta$ by $M$ such that $C^Z-B^Y=[(C+B)M\alpha-CBM\beta]^X$ will ensure equality, where
\begin{subequations}
\begin{gather}
C^Z-B^Y = [(C+B)M\alpha-CBM\beta]^X              \label{eqn:29a}\\
\sqrt[X]{C^Z-B^Y} = M[(C+B)\alpha-CB\beta]       \label{eqn:29b}\\
M= \frac{\sqrt[X]{C^Z-B^Y}}{(C+B)\alpha-CB\beta} \label{eqn:29c}
\end{gather}
\end{subequations}
Since every set $\{\alpha,\beta\}$ is unique, then per \cref{eqn:29c} there exists a unique $M$ that satisfies $C^Z-B^Y=[(C+B)M\alpha-CBM\beta]^X$ when $\displaystyle{\alpha=\sqrt[X]{C^{Z-X}-B^{Y-X}}}$ and $\displaystyle{\beta=\sqrt[X]{C^{Z-2X}-B^{Y-2X}}}$ simultaneously.

Given that $C^Z-B^Y$ and $[(C+B)M\alpha-CBM\beta]^X$ are identical, their binomial expansions are structurally identical, and their sums are identical, then indeed $\displaystyle{\alpha=\sqrt[X]{C^{Z-X}-B^{Y-X}}}$ and $\displaystyle{\beta=\sqrt[X]{C^{Z-2X}-B^{Y-2X}}}$ together ensure the equality of \cref{eqn:26,eqn:27b} hold for $M$ as defined in \cref{eqn:29c}.
\end{proof}

\bigskip

\noindent\textbf{Characteristics of $\bm{M}$, $\bm{\alpha}$ and $\bm{\beta$} from \cref{Thm:2.9_Real_Alpha_Beta}}

The important feature is not the scalar but instead the characteristics of $\alpha$ and $\beta$ as defined above despite the scalar.  Based on \BealsEq, we know $A=\sqrt[X]{C^Z-B^Y}$.  The structural similarity between $C^Z-B^Y$ in the formula for $A$ and the expressions $C^{Z-X}-B^{Y-X}$ and $C^{Z-2X}-B^{Y-2X}$ in the formulas for $\alpha$ and $\beta$, respectively, is critical.  This structural similarity will be explored and exploited later in this document.

We note $\alpha$ and $\beta$ could be defined differently from above and still maintain the equality of $C^Z-B^Y=[(C+B)\alpha-CB\beta]^X$, without scalar $M$.  However, if one defined $\alpha$ or $\beta$ differently than above, then either of these terms must be $\beta=\displaystyle{\frac{\sqrt[X]{C^Z-B^Y}-(C+B)\alpha}{-CB}}$ or $\alpha=\displaystyle{\frac{\sqrt[X]{C^Z-B^Y}+CB\beta}{C+B}}$ in order to satisfy $C^Z-B^Y=[(C+B)\alpha-CB\beta]^X$.  Since any arbitrary $\alpha$ corresponds to a unique $\beta$, there are an infinite number of sets $\{\alpha,\beta\}$ that satisfy this equation.  In some conditions, there are no rational pairs among the infinite number of sets $\{\alpha,\beta\}$, such as if $\sqrt[X]{C^Z-B^Y}$ is irrational.  But per the above, there is only unique set $\{\alpha,\beta\}$ when $\displaystyle{\alpha=\sqrt[X]{C^{Z-X}-B^{Y-X}}}$ and $\displaystyle{\beta=\sqrt[X]{C^{Z-2X}-B^{Y-2X}}}$ when scalar $M$ is applied accordingly.  We note all possible sets of $\{\alpha,\beta\}$ map to $\displaystyle{\alpha=\sqrt[X]{C^{Z-X}-B^{Y-X}}}$ and $\displaystyle{\beta=\sqrt[X]{C^{Z-2X}-B^{Y-2X}}}$ for scalar $M$ since all sets must satisfy $C^Z-B^Y=[(C+B)\alpha-CB\beta]^X$.

\bigskip

\begin{theorem}
\label{Thm:2.10_No_Solution_Alpha_Beta_Irrational}
\Conjecture, if there exists set $\{A,B,C,X,Y,Z\}$ that does not satisfy these conditions, then $\alpha$ or $\beta$ that satisfies $C^Z-B^Y=[(C+B)\alpha-CB\beta]^X$ will be irrational.
\end{theorem}
\begin{proof}
From \cref{Thm:2.9_Real_Alpha_Beta}, we know for any and every possible value of $\alpha$, that $\beta=\displaystyle{\frac{\sqrt[X]{C^Z-B^Y}-(C+B)\alpha}{-CB}}$.  Without loss of generality, suppose $B$, $C$, $X$, $Y$, and $Z$ are integer but there is no integer value of $A$ that satisfies the conjecture, then $\displaystyle{A=\sqrt[X]{C^Z-B^Y}}$ must be irrational.  Hence $\beta$ is irrational given the irrational term $\displaystyle{\sqrt[X]{C^Z-B^Y}}$ in the numerator of $\beta$.

We also know from \cref{Thm:2.9_Real_Alpha_Beta} given any and every possible value of $\beta$ that $\alpha=\displaystyle{\frac{\sqrt[X]{C^Z-B^Y}+CB\beta}{C+B}}$.  Here too without loss of generality, if $B$, $C$, $X$, $Y$, and $Z$ are integer but there is no integer value of $A$ that satisfies the conjecture, then $\displaystyle{A=\sqrt[X]{C^Z-B^Y}}$ must be irrational.  Hence $\alpha$ is irrational given the irrational term $\displaystyle{\sqrt[X]{C^Z-B^Y}}$ in the numerator of $\alpha$.

Hence when set $\{A,B,C,X,Y,Z\}$ does not satisfy the conjecture, then the corresponding $\alpha$ or $\beta$ that satisfies $C^Z-B^Y=[(C+B)\alpha-CB\beta]^X$ is irrational.
\end{proof}

We note that the exclusion of scalar $M$ from $C^Z-B^Y=[(C+B)M\alpha-CBM\beta]^X$ in \cref{Thm:2.10_No_Solution_Alpha_Beta_Irrational} (letting $M=1$) does not change the outcome in that if $C^Z-B^Y$ is not a perfect power, then scaled or unscaled $\alpha$ and $\beta$ cannot change the irrationality of the root of $C^Z-B^Y$.  Since $\displaystyle{\alpha=\sqrt[X]{C^{Z-X}-B^{Y-X}}}$, $\displaystyle{\beta=\sqrt[X]{C^{Z-2X}-B^{Y-2X}}}$, and scalar $M$ always together satisfy $C^Z-B^Y=[(C+B)M\alpha-CBM\beta]^X$, then we can study the properties of $\displaystyle{\alpha=\sqrt[X]{C^{Z-X}-B^{Y-X}}}$ and $\displaystyle{\beta=\sqrt[X]{C^{Z-2X}-B^{Y-2X}}}$ as related to key characteristics of the conjecture.  To do so, we need to establish the implication of coprimality and irrationality as it relates to $\alpha$ and $\beta$.

\bigskip

\begin{theorem}
\label{Thm:2.11_Coprime_Alpha_Beta_Irrational}
\Conjecture, if gcd$(A,B,C)=1$, then both $\displaystyle{\alpha=\sqrt[X]{C^{Z-X}-B^{Y-X}}}$ and $\displaystyle{\beta=\sqrt[X]{C^{Z-2X}-B^{Y-2X}}}$ are irrational.
\end{theorem}

\begin{proof}
Suppose $\displaystyle{\alpha=\sqrt[X]{C^{Z-X}-B^{Y-X}}}$ is rational, $A,B,C,X,Y,$ and $Z$ are integers that satisfy the conjecture, and gcd$(A,B,C)=1$.  Then we can express $\alpha$ as
\begin{subequations}
\begin{gather}
\alpha^X=C^{Z-X}-B^{Y-X}
\label{eqn:30a} \\
\alpha^X=\frac{C^Z}{C^X}-\frac{B^Y}{B^X}
\label{eqn:30b} \\
\alpha^X=\frac{B^XC^Z-C^XB^Y}{B^XC^X}
\label{eqn:30c}
\end{gather}
\end{subequations}
Since $\alpha$ is rational, then $\alpha$ can be expressed as the ratio of integers $p$ and $q$ such that
\begin{equation}
\frac{p^X}{q^X}=\frac{B^XC^Z-C^XB^Y}{B^XC^X}\label{eqn:31}
\end{equation}
where $p^X=B^XC^Z-C^XB^Y$ and $q^X=B^XC^X$.  We note that the denominator of \cref{eqn:31} is a perfect power of $X$ and thus we know that the numerator must also be a perfect power of $X$ for the  $X^{th}$ root of their ratio to be rational.  Hence the ratio of the perfect power of the numerator and the perfect power of the denominator, even after simplifying, must thus be rational.

Regardless of the parity of $B$ or $C$, per \cref{eqn:31}, $p^X$ must be even as it is the difference of two odd numbers or the difference of two even numbers.  Furthermore, since $p^X$ by definition is a perfect power of $X$ and given that it is even, then $p^X$ must be defined as $2^{iX}(f_1f_2 \cdots f_{n_f})^{jX}$ for positive integers $i$, $j$, $f_1$, $f_2$, ..., $f_{n_f}$ where $2^{iX}$ is the perfect power of the even component of $p^X$ and $(f_1f_2\cdots f_{n_f})^{jX}$ is the perfect power of the remaining $n_f$ prime factors of $p^X$ where $f_1$, $f_2$, ..., $f_{n_f}$ are the remaining prime factors of $p^X$.  Hence
\begin{equation}
\frac{p^X}{q}=\frac{2^{iX}(f_1f_2\cdots f_{n_f})^{jX}}{BC} = \frac{B^XC^Z-C^XB^Y}{BC}\label{eqn:32}
\end{equation}
$B$ and $C$ can also be expressed as a function of their prime factors, thus
\begin{equation}
\frac{p^X}{q}=\frac{2^{iX}f_1^{jX}f_2^{jX}\cdots f_{n_f}^{jX}}{b_1b_2\cdots b_{n_b} c_1c_2\cdots c_{n_c}}
\label{eqn:33}
\end{equation}
where $b_1, b_2, \dots,b_{n_b}$ and $c_1, c_2, \dots,c_{n_c}$ are prime factors of $B$ and $C$ respectively.  Based on the right side of \cref{eqn:32}, the entire denominator $BC$ is fully subsumed by the numerator, and thus every one of the prime factors in the denominator equals one of the prime factors in the numerator.  Thus after dividing, one or more of the exponents in the numerator reduces thereby canceling the entire denominator accordingly.  For illustration, suppose $b_1=b_2=2$, $b_{n_b}=f_1$, $c_1=c_2=f_2$ and $c_{n_c}=f_{n_f}$.  As such, \cref{eqn:33} simplifies to
\begin{equation}
p^X=2^{iX-2}f_1^{jX-1}f_2^{jX-2}\cdots f_{n_f}^{jX-1}
\label{eqn:34}
\end{equation}
which has terms with exponents that are not multiples of $X$ and are thus not perfect powers of $X$.  Therefore the $X^{th}$ root is irrational which contradicts the assumption that $\displaystyle{\alpha=\frac{p}{q}}$ is rational.  We note that all factors in the denominator of \cref{eqn:33} cannot each be perfect powers of $X$ since the bases $B$ and $C$ are defined to be reduced.  More generally, beyond the illustration, after simplifying one or more terms, \cref{eqn:34} will have an exponent that is not a multiple of $X$ and thus is irrational when taking the root accordingly.

Suppose $\displaystyle{\beta=\sqrt[X]{C^{Z-2X}-B^{Y-2X}}}$ is rational, $A,B,C,X,Y,$ and $Z$ are integers, and gcd$(A,B,C)=1$.  Then we can express $\beta$ as
\begin{subequations}
\begin{gather}
\beta^X=C^{Z-2X}-B^{Y-2X}
\label{eqn:35a} \\
\beta^X=\frac{C^Z}{C^{2X}}-\frac{B^Y}{B^{2X}}
\label{eqn:35b} \\
\beta^X=\frac{B^{2X}C^Z-C^{2X}B^Y}{B^{2X}C^{2X}}
\label{eqn:35c}
\end{gather}
\end{subequations}
Since $\beta$ is rational, then $\beta$ can be expressed as the ratio of integers $p$ and $q$ such that
\begin{equation}
\frac{p^X}{q^X}=\frac{B^{2X}C^Z-C^{2X}B^Y}{B^{2X}C^{2X}}\label{eqn:36}
\end{equation}
where $p^X=B^{2X}C^Z-C^{2X}B^Y$ and $q^X=B^{2X}C^{2X}$.  We note that the denominator of \cref{eqn:36} is a perfect power of $X$ and thus we know that the numerator must also be a perfect power of $X$ for the  $X^{th}$ root of their ratio to be rational.  Hence the ratio of the perfect power of the numerator and the perfect power of the denominator, even after simplifying, must thus be rational.

Regardless of the parity of $B$ or $C$, per \cref{eqn:36}, $p^X$ must be even as it is the difference of two odd numbers or the difference of two even numbers.  Furthermore, since $p^X$ by definition is a perfect power of $X$ and given that it is even, then $p^X$ must be defined as $2^{iX}(f_1f_2 \cdots f_{n_f})^{jX}$ for positive integers $i$, $j$, $f_1$, $f_2$, ..., $f_{n_f}$ where $2^{iX}$ is the perfect power of the even component of $p^X$ and $(f_1f_2\cdots f_{n_f})^{jX}$ is the perfect power of the remaining $n_f$ prime factors of $p^X$ where $f_1$, $f_2$, ..., $f_{n_f}$ are the remaining prime factors of $p^X$.  Hence
\begin{equation}
\frac{p^X}{q}=\frac{2^{iX}(f_1f_2\cdots f_{n_f})^{jX}}{B^2C^2} = \frac{B^{2X}C^Z-C^{2X}B^Y}{B^2C^2}\label{eqn:37}
\end{equation}
$B^2$ and $C^2$ can also be expressed as a function of their prime factors, thus
\begin{equation}
\frac{p^X}{q}=\frac{2^{iX}f_1^{jX}f_2^{jX}\cdots f_{n_f}^{jX}}{b_1^2b_2^2\cdots b_{n_b}^2 c_1^2c_2^2\cdots c_{n_c}^2}
\label{eqn:38}
\end{equation}
where $b_1, b_2, \dots,b_{n_b}$ and $c_1, c_2, \dots,c_{n_c}$ are prime factors of $B$ and $C$ respectively.  Based on the right side of \cref{eqn:37}, the entire denominator $B^2C^2$ is fully subsumed by the numerator, and thus every one of the prime factors in the denominator equals one of the prime factors in the numerator.  Thus after dividing, one or more of the exponents in the numerator reduces thereby canceling the entire denominator accordingly.  For illustration, suppose $b_1=b_2=2$, $b_{n_b}=f_1$, $c_1=c_2=f_2$ and $c_{n_c}=f_{n_f}$.  As such, \cref{eqn:38} simplifies to
\begin{equation}
p^X=2^{iX-4}f_1^{jX-2}f_2^{jX-4}\cdots f_{n_f}^{jX-2}
\label{eqn:39}
\end{equation}
which has terms with exponents that are not multiples of $X$ and are thus not perfect powers of $X$.  Therefore the $X^{th}$ root is irrational which contradicts the assumption that $\displaystyle{\alpha=\frac{p}{q}}$ is rational.  We note that all factors in the denominator of \cref{eqn:38} cannot each be perfect powers of $X$ since the bases $B$ and $C$ are defined to be reduced.  More generally, beyond the illustration, after simplifying one or more terms, \cref{eqn:39} will have an exponent that is not a multiple of $X$ and thus is irrational when taking the root accordingly.

Since the definition of a rational number is the ratio of two integers in which the ratio is reduced, given that gcd$(B,C)=1$, then both $\displaystyle{\alpha=\sqrt[X]{C^{Z-X}-B^{Y-X}}}$ and $\displaystyle{\beta=\sqrt[X]{C^{Z-2X}-B^{Y-2X}}}$ are irrational.
\end{proof}

\bigskip

\cref{Thm:2.11_Coprime_Alpha_Beta_Irrational} establishes that if gcd$(A,B,C)=1$ then $\displaystyle{\alpha=\sqrt[X]{C^{Z-X}-B^{Y-X}}}$ and $\displaystyle{\beta=\sqrt[X]{C^{Z-2X}-B^{Y-2X}}}$ are irrational.  We now establish the reverse, such that if $\displaystyle{\alpha=\sqrt[X]{C^{Z-X}-B^{Y-X}}}$ and $\displaystyle{\beta=\sqrt[X]{C^{Z-2X}-B^{Y-2X}}}$ are rational then gcd$(A,B,C)>1$.

\bigskip

\begin{theorem}
\label{Thm:2.12_Rational_Alpha_Beta_Rational_Then_Not_Coprime}
\Conjecture, if $\displaystyle{\sqrt[X]{C^{Z-X}-B^{Y-X}}}$ or $\displaystyle{\sqrt[X]{C^{Z-2X}-B^{Y-2X}}}$ are rational, then gcd$(A,B,C)>1$.
\end{theorem}

\begin{proof}
Suppose $\displaystyle{\alpha=\sqrt[X]{C^{Z-X}-B^{Y-X}}}$ is rational and $A,B,C,X,Y,$ and $Z$ are integers that satisfy the conjecture.  Then we can express $\alpha$ as
\begin{subequations}
\begin{gather}
\alpha^X=C^{Z-X}-B^{Y-X}
\label{eqn:40a} \\
\alpha^X=\frac{C^Z}{C^X}-\frac{B^Y}{B^X}
\label{eqn:40b} \\
\alpha^X=\frac{B^XC^Z-C^XB^Y}{B^XC^X}
\label{eqn:40c}
\end{gather}
\end{subequations}
Since $\alpha$ is rational, then $\alpha$ can be expressed as the ratio of integers $p$ and $q$ such that
\begin{equation}
\frac{p^X}{q^X}=\frac{B^XC^Z-C^XB^Y}{B^XC^X}\label{eqn:41}
\end{equation}
where $p^X=B^XC^Z-C^XB^Y$ and $q^X=B^XC^X$.  We note that the denominator of \cref{eqn:41} is a perfect power of $X$ and thus we know that the numerator must also be a perfect power of $X$ for the $X^{th}$ root of their ratio to be rational. Hence the ratio of the perfect power of the numerator and the perfect power of the denominator, even after simplifying, must thus be rational.

Suppose gcd$(A,B,C)=k$ where integer $k\geq 2$.  Thus $A=ak$, $B=bk$, and $C=ck$ for pairwise coprime integers $a$, $b$, and $c$.  We can express \cref{eqn:41} with the common term, namely
\begin{subequations}
\begin{gather}
\frac{p^X}{q^X}=\frac{(kb)^X(kc)^Z-(kc)^X(kb)^Y}{(kb)^X(kc)^X}    \label{eqn:42a} \\
\frac{p^X}{q^X}=\frac{k^{X+Z}b^Xc^Z-k^{X+Y}c^Xb^Y}{k^{2X}b^Xc^X}  \label{eqn:42b} \\
\frac{p^X}{q^X}=\frac{k^{Z-X}b^Xc^Z-k^{Y-X}c^Xb^Y}{b^Xc^X}        \label{eqn:42c} \\
\frac{p^X}{q^X}=\frac{k^{min(Z-X,Y-X)}[k^{Z-min(Z-X,Y-X)}b^Xc^Z-k^{Y-min(Z-X,Y-X)}c^Xb^Y]}{b^Xc^X}                    \label{eqn:42d}
\end{gather}
\end{subequations}

Regardless of the parity of $b$, $c$, or $k$ per \cref{eqn:42c,eqn:42d}, $p^X$ must be even as it is the difference of two odd numbers or the difference of two even numbers.  Furthermore, since $p^X$ by definition is a perfect power of $X$ and given that it is even, then $p^X$ must be defined as $2^{iX}k^{min(Z-X,Y-X)}(f_1f_2 \cdots f_{n_f})^{jX}$ for positive integers $i$, $j$, $f_1$, $f_2$, ..., $f_{n_f}$ where $2^{iX}$ is the perfect power of the even component of $p^X$, $k^{min(Z-X,Y-X)}$ is the common factor based on gcd$(A,B,C)$, and $(f_1f_2\cdots f_{n_f})^{jX}$ is the perfect power of the remaining $n_f$ prime factors of $p^X$ where $f_1$, $f_2$, ..., $f_{n_f}$ are the remaining prime factors of $p^X$.  Hence
\begin{equation}
\frac{p^X}{q}=\frac{2^{iX}k^{min(Z-X,Y-X)}(f_1f_2\cdots f_n)^{jX}}{bc}
=\frac{k^{Z-X}b^Xc^Z-k^{Y-X}c^Xb^Y}{bc} \label{eqn:43}
\end{equation}
Both $b$ and $c$ can also be expressed as a function of their prime factors, thus
\begin{equation}
\frac{p^X}{q}=\frac{2^{iX}k^{min(Z-X,Y-X)}f_1^{jX}f_2^{jX}\cdots f_{n_f}^{jX}}{b_1b_2\cdots b_{n_b} c_1c_2\cdots c_{n_c}}
\label{eqn:44}
\end{equation}
where $b_1, b_2, \dots,b_{n_b}$ and $c_1, c_2, \dots,c_{n_c}$ are prime factors of $b$ and $c$ respectively.  Based on the right side of \cref{eqn:43}, the entire denominator $bc$ is fully subsumed by the numerator, and thus every one of the prime factors in the denominator equals one of the prime factors in the numerator.  Thus after dividing, one or more of the exponents in the numerator reduces thereby canceling the entire denominator accordingly.  For illustration, suppose $b_1=b_2=2$, $b_{n_b}=f_1$, $c_1=c_2=f_2$ and $c_{n_c}=f_{n_f}$.  As such, \cref{eqn:44} simplifies to
\begin{equation}
p^X=2^{iX-4}k^{min(Z-X,Y-X)}f_1^{jX-2}f_2^{jX-4}\cdots f_{n_f}^{jX-2}
\label{eqn:45}
\end{equation}
which has terms with exponents that are not multiples of $X$ and are thus not perfect powers of $X$.  If $k=1$, then the $X^{th}$ root is irrational which contradicts the assumption that $\displaystyle{\alpha=\frac{p}{q}}$ is rational. However if $k>1$ such that it is a composite of the factors that are not individually perfect powers of $X$, then the resulting expression is a perfect power of X.  Hence when $k=1$, then per \cref{Thm:2.11_Coprime_Alpha_Beta_Irrational},  $\displaystyle{\alpha=\sqrt[X]{C^{Z-X}-B^{Y-X}}}$ is irrational.  However if $\displaystyle{\alpha=\sqrt[X]{C^{Z-X}-B^{Y-X}}}$ is rational, then $k\neq1$ and thus gcd$(A,B,C)\neq1$.

\bigskip

Suppose $\displaystyle{\beta=\sqrt[X]{C^{Z-2X}-B^{Y-2X}}}$ is rational and $A,B,C,X,Y,$ and $Z$ are integers that satisfy the conjecture.  Then we can express $\beta$ as
\begin{subequations}
\begin{gather}
\beta^X=C^{Z-2X}-B^{Y-2X}
\label{eqn:46a} \\
\beta^X=\frac{C^Z}{C^2X}-\frac{B^Y}{B^2X}
\label{eqn:46b} \\
\beta^X=\frac{B^{2X}C^Z-C^{2X}B^Y}{B^{2X}C^{2X}}
\label{eqn:46c}
\end{gather}
\end{subequations}
Since $\beta$ is rational, then $\beta$ can be expressed as the ratio of integers $p$ and $q$ such that
\begin{equation}
\frac{p^X}{q^X}=\frac{B^{2X}C^Z-C^{2X}B^Y}{B^{2X}C^{2X}}\label{eqn:47}
\end{equation}
where $p^X=B^{2X}C^Z-C^{2X}B^Y$ and $q^X=B^{2X}C^{2X}$.  We note that the denominator of \cref{eqn:47} is a perfect power of $X$ and thus we know that the numerator must also be a perfect power of $X$ for the  $X^{th}$ root of their ratio to be rational.  Hence the ratio of the perfect power of the numerator and the perfect power of the denominator, even after simplifying, must thus be rational.

Suppose gcd$(A,B,C)=k$ where integer $k\geq 2$.  Thus $A=ak$, $B=bk$, and $C=ck$ for pairwise coprime integers $a$, $b$, and $c$.  We can express \cref{eqn:47} with the common term, namely
\begin{subequations}
\begin{gather}
\frac{p^X}{q^X}=\frac{(kb)^{2X}(kc)^Z-(kc)^{2X}(kb)^Y}{(kb)^{2X}(kc)^{2X}}\label{eqn:48a} \\
\frac{p^X}{q^X}=\frac{k^{2X+Z}b^Xc^Z-k^{2X+Y}c^Xb^Y}{k^{4X}b^Xc^X}        \label{eqn:48b} \\
\frac{p^X}{q^X}=\frac{k^{Z-2X}b^Xc^Z-k^{Y-2X}c^Xb^Y}{b^Xc^X}              \label{eqn:48c} \\
\frac{p^X}{q^X}=\frac{k^{min(Z-2X,Y-2X)}[k^{Z-min(Z-2X,Y-2X)}b^Xc^Z-k^{Y-min(Z-2X,Y-2X)}c^Xb^Y]}{b^Xc^X}                   \label{eqn:48d}
\end{gather}
\end{subequations}

Regardless of the parity of $b$, $c$, or $k$ per \cref{eqn:48c,eqn:48d}, $p^X$ must be even as it is the difference of two odd numbers or the difference of two even numbers.  Furthermore, since $p^X$ by definition is a perfect power of $X$ and given that it is even, then $p^X$ must be defined as $2^{iX}k^{min(Z-2X,Y-2X)}(f_1f_2 \cdots f_{n_f})^{jX}$ for positive integers $i$, $j$, $f_1$, $f_2$, ..., $f_{n_f}$ where $2^{iX}$ is the perfect power of the even component of $p^X$, $k^{min(Z-X,Y-X)}$ is the common factor based on gcd$(A,B,C)$, and $(f_1f_2\cdots f_{n_f})^{jX}$ is the perfect power of the remaining $n_f$ prime factors of $p^X$ where $f_1$, $f_2$, ..., $f_{n_f}$ are the remaining prime factors of $p^X$.  Hence
\begin{equation}
\frac{p^X}{q}=\frac{2^{iX}k^{min(Z-2X,Y-2X)}(f_1f_2\cdots f_{n_f})^{jX}}{bc}
=\frac{k^{Z-2X}b^Xc^Z-k^{Y-2X}c^Xb^Y}{bc} \label{eqn:49}
\end{equation}
Both $b$ and $c$ can also be expressed as a function of their prime factors, thus
\begin{equation}
\frac{p^X}{q}=\frac{2^{iX}k^{min(Z-2X,Y-2X)}f_1^{jX}f_2^{jX}\cdots f_{n_f}^{jX}}{b_1b_2\cdots b_{n_b} c_1c_2\cdots c_{n_c}}
\label{eqn:50}
\end{equation}
where $b_1, b_2, \dots,b_{n_b}$ and $c_1, c_2, \dots,c_{n_c}$ are prime factors of $b$ and $c$ respectively.  Based on the right side of \cref{eqn:49}, the entire denominator $bc$ is fully subsumed by the numerator, and thus every one of the prime factors in the denominator equals one of the prime factors in the numerator.  Thus after dividing, one or more of the exponents in the numerator reduces thereby canceling the entire denominator accordingly.  For illustration, suppose $b_1=b_2=2$, $b_{n_b}=f_1$, $c_1=c_2=f_2$ and $c_{n_c}=f_{n_f}$.  As such, \cref{eqn:50} simplifies to
\begin{equation}
p^X=2^{iX-4}k^{min(Z-2X,Y-2X)}f_1^{jX-2}f_2^{jX-4}\cdots f_{n_f}^{jX-2}
\label{eqn:51}
\end{equation}
which has terms with exponents that are not multiples of $X$ and are thus not perfect powers of $X$.  If $k=1$, then the $X^{th}$ root is irrational which contradicts the assumption that $\displaystyle{\alpha=\frac{p}{q}}$ is rational. However if $k>1$ such that it is a composite of the factors that are not individually perfect powers of $X$, then the resulting expression is a perfect power of X.  Hence when $k=1$, then per \cref{Thm:2.11_Coprime_Alpha_Beta_Irrational},  $\displaystyle{\beta=\sqrt[X]{C^{Z-2X}-B^{Y-2X}}}$ is irrational.  However if $\displaystyle{\beta=\sqrt[X]{C^{Z-2X}-B^{Y-2X}}}$ is rational, then $k\neq1$ and thus gcd$(A,B,C)\neq1$.

Thus if either or both $\displaystyle{\sqrt[X]{C^{Z-X}-B^{Y-X}}}$ or $\displaystyle{\sqrt[X]{C^{Z-2X}-B^{Y-2X}}}$ are rational, then gcd$(A,B,C)>1$.
\end{proof}

\bigskip

The values of $\displaystyle{\alpha=\sqrt[X]{C^{Z-X}-B^{Y-X}}}$ and $\displaystyle{\beta=\sqrt[X]{C^{Z-2X}-B^{Y-2X}}}$ have critical properties as related to coprimality and other relationships with integer solutions that satisfy the conjecture.  We know from \cref{Thm:2.11_Coprime_Alpha_Beta_Irrational} that if gcd$(A,B,C)=1$, then $\displaystyle{\alpha=\sqrt[X]{C^{Z-X}-B^{Y-X}}}$ and $\displaystyle{\beta=\sqrt[X]{C^{Z-2X}-B^{Y-2X}}}$ are both irrational.  Even though all feasible values of $\alpha$ and $\beta$ map to this pair given \cref{Thm:2.9_Real_Alpha_Beta}, we can still consider other feasible values of $\alpha$ and $\beta$ as related to coprimality.

\bigskip

\begin{theorem}
\label{Thm:2.13_Coprime_Any_Alpha_Beta_Irrational_Indeterminate}
\Conjecture, if gcd$(A,B,C)=1$, then any value of $\alpha$ or $\beta$ that satisfies $C^Z-B^Y=[(C+B)\alpha-CB\beta]^X$ must be irrational or indeterminate.
\end{theorem}

\begin{proof}
Given $\sqrt[X]{C^Z-B^Y} =(C+B)\alpha-CB\beta$ per \cref{Thm:2.8_Functional_Form}.  If we suppose $\sqrt[X]{C^Z-B^Y}$ is irrational, then $(C+B)\alpha-CB\beta$ is irrational.  With $C$ and $B$ integer, if $\alpha$ and $\beta$ were rational, then $(C+B)\alpha-CB\beta$ is composed solely of rational terms and thus must be rational, which contradicts the assumption of irrationality. Hence $\alpha$ or $\beta$ must be irrational.  Further, per \cref{Thm:2.11_Coprime_Alpha_Beta_Irrational}, $\alpha$ and $\beta$ are irrational if $\sqrt[X]{C^Z-B^Y}$ is irrational and gcd$(A,B,C)=1$.  Thus here too $\alpha$ or $\beta$ must be irrational.

If we suppose instead $\sqrt[X]{C^Z-B^Y}$ is rational, then $(C+B)\alpha-CB\beta$ is rational.
Given gcd$(A,B,C)=1$, per \cref{Thm:2.11_Coprime_Alpha_Beta_Irrational} we know both $\displaystyle{\alpha=\sqrt[X]{C^{Z-X}-B^{Y-X}}}$ and $\displaystyle{\beta=\sqrt[X]{C^{Z-2X}-B^{Y-2X}}}$ are irrational.

Suppose instead $\alpha$ is to be defined as any real other than $\displaystyle{\sqrt[X]{C^{Z-X}-B^{Y-X}}}$.  As such, per \cref{Thm:2.8_Functional_Form}, $\beta$ is derived by $\beta=\displaystyle{\frac{\sqrt[X]{C^Z-B^Y}-(C+B)\alpha}{-CB}}$.  Hence substituting into \cref{eqn:24} gives us
\begin{subequations}
\begin{align}
\sqrt[X]{C^Z-B^Y} &=(C+B)\alpha-CB\beta                        \label{eqn:52a} \\
\sqrt[X]{C^Z-B^Y} &=(C+B)\alpha-CB\frac{\sqrt[X]{C^Z-B^Y}-(C+B)\alpha}{-CB} \label{eqn:52b}\\
\sqrt[X]{C^Z-B^Y} &=(C+B)\alpha+\sqrt[X]{C^Z-B^Y}-(C+B)\alpha    \label{eqn:52c} \\
\sqrt[X]{C^Z-B^Y} &=\sqrt[X]{C^Z-B^Y}                            \label{eqn:52d}
\end{align}
\end{subequations}
Regardless of the value selected for $\alpha$, both $\alpha$ and $\beta$ fall out and thus both are indeterminate when gcd$(A,B,C)=1$.

Suppose instead $\beta$ is to be defined as any real other than $\displaystyle{\sqrt[X]{C^{Z-2X}-B^{Y-2X}}}$.  As such, per \cref{Thm:2.8_Functional_Form}, $\alpha$ is derived by $\alpha=\displaystyle{\frac{\sqrt[X]{C^Z-B^Y}+CB\beta}{C+B}}$.  Hence substituting into \cref{eqn:24} gives us
\begin{subequations}
\begin{align}
\sqrt[X]{C^Z-B^Y} &=(C+B)\alpha-CB\beta                 \label{eqn:53a} \\
\sqrt[X]{C^Z-B^Y} &=(C+B)\frac{\sqrt[X]{C^Z-B^Y}+CB\beta}{C+B}-CB\beta  \label{eqn:53b}\\
\sqrt[X]{C^Z-B^Y} &=\sqrt[X]{C^Z-B^Y}+CB\beta-CB\beta     \label{eqn:53c} \\
\sqrt[X]{C^Z-B^Y} &=\sqrt[X]{C^Z-B^Y}                     \label{eqn:53d}
\end{align}
\end{subequations}
Regardless of the value selected for $\beta$, both $\alpha$ and $\beta$ fall out and thus both are indeterminate when gcd$(A,B,C)=1$.

Hence if gcd$(A,B,C)=1$, then both of the terms $\displaystyle{\alpha=\sqrt[X]{C^{Z-X}-B^{Y-X}}}$ and $\displaystyle{\beta=\sqrt[X]{C^{Z-2X}-B^{Y-2X}}}$ must be irrational, while any other random value of $\alpha$ or $\beta$ that satisfies $C^Z-B^Y=[(C+B)\alpha-CB\beta]^X$ must be irrational or indeterminate.
\end{proof}
\label{Section:Reparameterize_End}

\bigskip

\Line
\subsection{Impossibility of the Terms}
\label{Section:Impossibility_Start}
Having established that pairwise coprimality is a definite byproduct when gcd$(A,B,C)=1$  (\cref{Thm:2.2_Coprime,Thm:2.3_Coprime,Thm:2.4_Coprime}), there is a unique reparamterization of the bases of \BealsEq\, whose rationality is tied to coprimality of terms (\cref{Thm:2.9_Real_Alpha_Beta,Thm:2.10_No_Solution_Alpha_Beta_Irrational,Thm:2.11_Coprime_Alpha_Beta_Irrational,Thm:2.12_Rational_Alpha_Beta_Rational_Then_Not_Coprime}), and that a line through the origin with an irrational slope does not go through any non-trivial lattice points (\cref{Thm:2.1_Irrational_Slope_No_Lattice}), we now delve into proving the conjecture under two mutually exclusive conditions:
\begin{enumerate}
\item geometric implications when gcd$(A,B,C)=1$
\item geometric implications when gcd$(A,B,C)>1$
\end{enumerate}
These steps lead to a critical contradiction which demonstrates the impossibility of the existance of counter-examples due to fundamental features of the conjecture.

As implied by Catalan's conjecture and proven by Mihăilescu \cite{mihailescu2004primary}, no integer solutions exist when $A$, $B$, or $C$ equals 1, regardless of coprimality.  Hence we consider the situation in which $A,B,C\geq2$.  Given the configuration of the conjecture, $A$, $B$, $C$, $X$, $Y$, and $Z$ are positive integers, and thus $A^X$, $B^Y$, and $C^Z$ are also integers.  A set of values that satisfy the conjecture can be plotted on a Cartesian coordinate grid with axes $A^X$, $B^Y$, and $C^Z$.  See \cref{Fig:3DScatter}.  Based on \cref{eqn:1} the line passing through the origin and the point $(A^X,B^Y,A^X+B^Y)$ can be expressed based on the angles in relation to the axes (see \cref{Fig:ScatterPlotWithAngles}), namely
\begin{subequations}
\begin{gather}
\theta_{_{C^ZB^Y}} = \tan^{-1}\frac{A^X+B^Y}{B^Y} \label{eqn:54a} \\
\theta_{_{C^ZA^X}} = \tan^{-1}\frac{A^X+B^Y}{A^X} \label{eqn:54b} \\
\theta_{_{B^YA^X}} = \tan^{-1}\frac{B^Y}{A^X}     \label{eqn:54c}
\end{gather}
\end{subequations}
where $\theta_{_{C^ZB^Y}}$ is the angle subtended between the $B^Y$ axis to the line through the origin and the given point in the $C^Z \times B^Y$ plane, $\theta_{_{C^ZA^X}}$ is the angle subtended between the $A^X$ axis to the line through the origin to the given point in the $C^Z \times A^X$ plane, and $\theta_{_{B^YA^X}}$ is the angle subtended between the $B^Y$ axis to the line through the origin to the given point in the $B^Y \times A^X$ plane.

The line subtended in each plane based on the origin and the given point $(A^X,B^Y,A^X+B^Y)$ has slopes that by definition are identical to the arguments in the corresponding tangent functions in \crefrange{eqn:54a}{eqn:54c}. In each case, the numerator and denominator are integers, and thus the corresponding ratios (and therefore slopes) are rational.  Given that this line passes through the origin and has a rational slope in all three planes, then we conclude the infinitely long line passes through infinitely many lattice points, namely at specific integer multiples of the slopes.

\begin{figure}
\includegraphics[width=.55\textwidth]{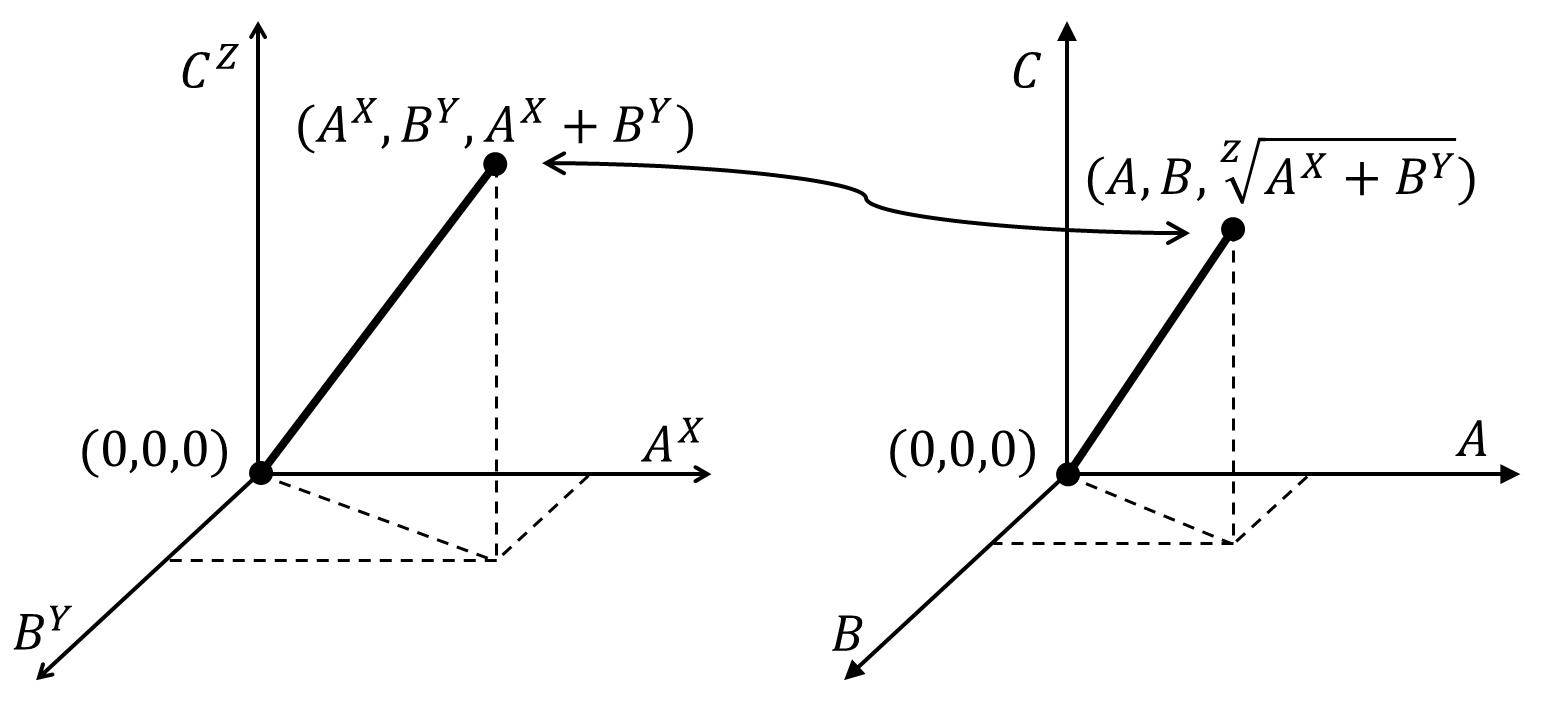}
\caption{Map of a corresponding point between two different plots based on \BealsEq\, satisfying the associated constraints and bounds.}
\label{Fig:3DScatterMap}
\end{figure}

Based on the conjecture the given lattice point $(A^X,B^Y,A^X+B^Y)$ relative to axes $A^X$, $B^Y$, and $C^Z$ corresponds to a lattice point $(A,B,\sqrt[Z]{A^X+B^Y})$ in a scatter plot based on axes $A$, $B$, and $C$.  See \cref{Fig:3DScatterMap}.  The conjecture states there is no integer solution to \BealsEq\, that simultaneously satisfies all the conditions if gcd$(A,B,C)=1$.  Thus from a geometric perspective, this means that if gcd$(A,B,C)=1$, then the line in the scatter plot based on axes $A$, $B$, and $C$ in \cref{Fig:3DScatterMap} could never go through a non-trivial lattice point since $A$, $B$, and $C$ could not all be integer simultaneously.  Conversely, if the corresponding line in the scatter plot based on axes $A$, $B$, and $C$ in \cref{Fig:3DScatterMap} does go through a non-trivial lattice point, then based on the conjecture we know gcd$(A,B,C)>1$.  Hence to validate the conjecture, we need to test the relationship between \BealsEq, the lattice points in both graphs, and the slopes subtended by the origin and these lattice points.

\begin{figure}
\includegraphics[width=.66\textwidth]{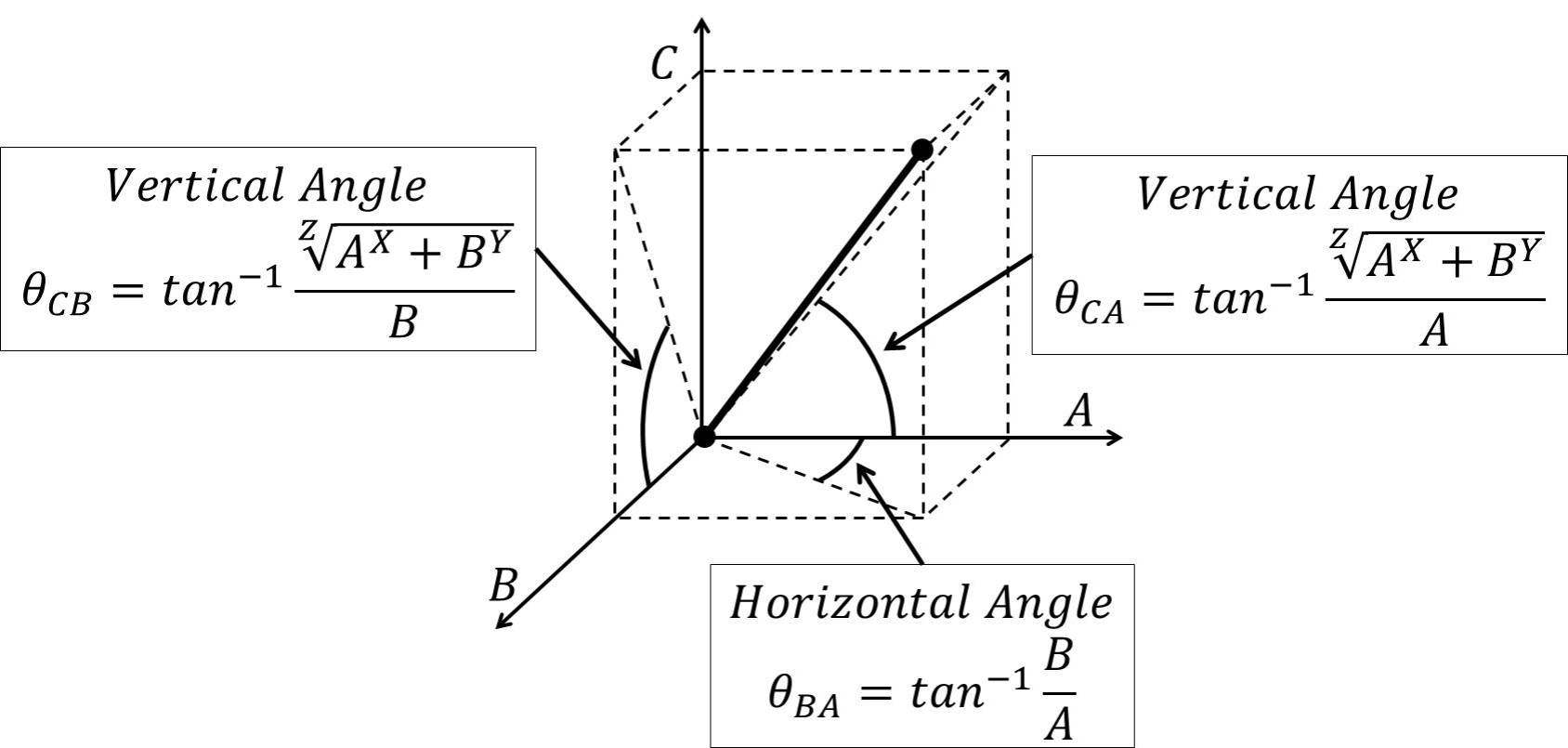}
\caption{Angles between the axes and the line segment subtended by the origin and a single point $A$, $B$, and $C$ from \cref{Fig:CZBYAZScatter}.}
\label{Fig:ScatterPlotWithAngles2}
\end{figure}

\bigskip

\begin{theorem}
\label{Thm:2.14_Main_Proof_Coprime_No_Solutions}
\Conjecture, if $gcd(A,B,C) = 1$, then there is no integer solution.
\end{theorem}

\begin{proof}
The line through the origin and point $(A,B,\sqrt[Z]{A^X+B^Y})$ can be expressed based on the angles in relation to the axes.  See \cref{Fig:ScatterPlotWithAngles2}.  These angles are
\begin{subequations}
\begin{gather}
\theta_{_{CB}} = \tan^{-1}\frac{\sqrt[Z]{A^X+B^Y}}{B} \label{eqn:55a} \\
\theta_{_{CA}} = \tan^{-1}\frac{\sqrt[Z]{A^X+B^Y}}{A} \label{eqn:55b} \\
\theta_{_{BA}} = \tan^{-1}\frac{B}{A}                 \label{eqn:55c}
\end{gather}
\end{subequations}
where $\theta_{_{CB}}$ is the angle subtended between the $B$ axis to the line through the origin and the given point in the $C\times B$ plane, $\theta_{_{CA}}$ is the angle subtended between the $A$ axis to the line through the origin to the given point in the $C\times A$ plane, and $\theta_{_{BA}}$ is the angle subtended between the $B$ axis to the line through the origin to the given point in the $B\times A$ plane.  See \cref{Fig:ScatterPlotWithAngles2}.

The line that corresponds to $\theta_{_{CB}}$ in \cref{eqn:55a} has slope $\displaystyle{m=\frac{\sqrt[Z]{A^X+B^Y}}{B}}$ in the $C\times B$ plane, and the line that corresponds to $\theta_{_{CA}}$ in \cref{eqn:55b} has slope $\displaystyle{m=\frac{\sqrt[Z]{A^X+B^Y}}{A}}$ in the $C\times A$ plane.  These two slopes are different than the slope of the line that corresponds to $\theta_{_{BA}}$ in \cref{eqn:55c} which is $\displaystyle{m=\frac{B}{A}}$ in the $B\times A$ plane since this latter slope is merely the ratio of two integers whereas the numerator of the first two are $\sqrt[Z]{A^X+B^Y}$ which may not be rational.

Building from \cref{eqn:55a,eqn:55b}, let $\MCB$ and $\MCA$
be the slopes of the lines through the origin and the given point in the $C \times B$ and $C \times A$ planes, respectively.  Thus we have
\begin{subequations}
\begin{align}
\MCB  &= \frac{\sqrt[Z]{A^X+B^Y}}{B}    &
\MCA  &= \frac{\sqrt[Z]{A^X+B^Y}}{A}
\label{eqn:56a} \\
\MCB  &= \sqrt[Z]{\frac{A^X+B^Y}{B^Z}}    &
\MCA  &= \sqrt[Z]{\frac{A^X+B^Y}{A^Z}}
\label{eqn:56b} \\
\MCB  &= \left(\frac{A^X}{B^Z} +B^{Y-Z}\right) ^{\frac{1}{Z}}   &
\MCA  &= \left(A^{X-Z} + \frac{B^Y}{A^Z} \right) ^{\frac{1}{Z}}
\label{eqn:56c}
\end{align}
\end{subequations}
Per \cref{Thm:2.1_Irrational_Slope_No_Lattice}, if a line through the origin has an irrational slope, then that line does not pass through any non-trivial lattice points.  Relative to the conjecture, a line in 3 dimensions that passes through the origin also passes through a point equal to the integer bases which satisfy the terms of the conjecture.  Since the solution that satisfies the conjecture is integer and must be a lattice point, then the corresponding lines must have rational slopes.  Hence if there exist integer solutions, then slopes $\MCB$ and $\MCA$ are rational.  If the slopes are irrational, then there is no integer solution.  We must now consider three mutually exclusive scenarios in relation to slopes $\MCB$ and $\MCA$:

\bigskip

\textbf{Scenario 1 of 2:  $\bm{\displaystyle{A^X \neq B^Y}$}.} Suppose $A^X\neq B^Y$.  Dividing both terms by $B^Z$ gives us $\displaystyle{\frac{A^X}{B^Z} \neq B^{Y-Z}}$.  Likewise, dividing both terms of $A^X \neq B^Y$ by $A^Z$ gives us $\displaystyle{A^{X-Z} \neq \frac{B^Y}{A^Z}}$.  Using this relationship, $\MCB$ and  $\MCA$ from \cref{eqn:56c} become
\begin{subequations}
\begin{align}
\MCB &=  \left(\frac{A^X}{B^Z}\right)^{\frac{1}{Z}} \left(1+ \frac{B^{Y-Z}}
         {\left(\frac{A^X}{B^Z}\right)}\right)^{\frac{1}{Z}}  &
\MCA &=  (A^{X-Z})^{\frac{1}{Z}} \left(1+ \frac{\left(\frac{B^Y}{A^Z}\right)}
         {A^{X-Z}} \right)^{\frac{1}{Z}}
\label{eqn:57a} \\
\MCB &=  \frac{A^{\frac{X}{Z}}}{B} \left(1+ \frac{B^Y}{A^X}\right)^{\frac{1}{Z}}  &
\MCA &=  ({A^{\frac{X}{Z}-1}}) \left(1+ \frac{B^Y}{A^X}\right)^{\frac{1}{Z}}
\label{eqn:57b}
\end{align}
\end{subequations}

Both $\MCB$ and $\MCA$ must be rational for there to exist an integer solution that satisfies the conjecture.  Based on \cref{eqn:57b}, there are two cases that can ensure both $\MCB$ and $\MCA$ are rational:
\begin{enumerate}
\item The term $\displaystyle{\left(1+\frac{B^Y}{A^X}\right)^{\frac{1}{Z}}}$ from
      \cref{eqn:57b} is rational, therefore both terms
      $\displaystyle{\frac{A^{\frac{X}{Z}}}{B}}$ and $A^{\frac{X}{Z}-1}$ must be rational so
      their respective products are rational.
\item The term $\displaystyle{\left(1+\frac{B^Y}{A^X}\right)^{\frac{1}{Z}}}$ from
      \cref{eqn:57b} is irrational, therefore both terms
      $\displaystyle{\frac{A^{\frac{X}{Z}}}{B}}$ and $A^{\frac{X}{Z}-1}$ are irrational.
      However the irrationality of these two terms are canceled with the irrationality
      of the denominator of $\displaystyle{\left(1+\frac{B^Y}{A^X}\right)^{\frac{1}{Z}}}$
      so their respective products are rational.
\end{enumerate}

\bigskip

\noindent\textbf{Case 1 of 2: the terms of $\bm{\MCB }$ and $\bm{\MCA }$ are rational}

Starting with the first case, assume $\displaystyle{\left(1+ \frac{B^Y}{A^X}\right)^{\frac{1}{Z}}}$ in \cref{eqn:57b} is rational and thus $\displaystyle{\frac{A^{\frac{X}{Z}}}{B}}$ and $A^{\frac{X}{Z}-1}$ are rational.  Since $A$ is reduced (not a perfect power), gcd$(A,B)=1$, and $A$ is raised to exponent $\displaystyle{\frac{X}{Z}}$ and $\displaystyle{\frac{X}{Z}-1}$, respectively, these terms are rational only if $X$ is an integer multiple of $Z$.  However, per \cref{Dfn:2.1_X_cannot_be_mult_of_Z} on page \pageref{Dfn:2.1_X_cannot_be_mult_of_Z}\, $X$ is not an integer multiple of $Z$, therefore the requirements to ensure $\MCB$ and $\MCA$  are rational cannot be met.

\bigskip
\noindent\textbf{Case 2 of 2: the terms of $\bm{\MCB }$ and $\bm{\MCA }$ are irrational}

Considering the second case, assume $\displaystyle{\left(1+ \frac{B^Y}{A^X}\right)^{\frac{1}{Z}}}$ in \cref{eqn:57b} is irrational, and thus $\displaystyle{\frac{A^{\frac{X}{Z}}}{B}}$ and $A^{\frac{X}{Z}-1}$ must also be irrational such that they cancel the irrationality with multiplication.  Since slopes $\MCB$ and $\MCA$ in \cref{eqn:57b} are rational with irrational terms, both \cref{eqn:57a,eqn:57b} must be rational.  We can re-express \cref{eqn:56a} equivalently as
\begin{equation}
\MCB = \frac{C}{\sqrt[Y]{C^Z-A^X}}   \qquad\qquad\quad
\MCA = \frac{C}{\sqrt[X]{C^Z-B^Y}}
\label{eqn:58}
\end{equation}
in which the denominators must be rational.  Per \cref{Thm:2.8_Functional_Form}, the denominators in \cref{eqn:58} can be reparameterized and thus \cref{eqn:58} becomes
\begin{equation}
        \MCB = \frac{C}{(C+A)M_{_1}\alphaCA  -CAM_{_1}\betaCA }   \qquad\qquad
        \MCA = \frac{C}{(C+B)M_{_2}\alphaCB  -CBM_{_2}\betaCB }
        \label{eqn:59}
\end{equation}
where $\alphaCA$, $\alphaCB$, $\betaCA$, and $\betaCB$ are positive rational numbers and $M_{_1}$ and $M_{_2}$ are positive scalars.   Per \cref{Thm:2.8_Functional_Form,Thm:2.9_Real_Alpha_Beta} $\alphaCA$, $\alphaCB$, $\betaCA$ and $\betaCB$ can be defined as
\begin{subequations}
\begin{align}
\alphaCA &=\sqrt[Y]{C^{Z-Y}-A^{X-Y}} &
\alphaCB &=\sqrt[X]{C^{Z-X}-B^{Y-X}}
\label{eqn:60a} \\
\betaCA  &=\sqrt[Y]{C^{Z-2Y}-A^{X-2Y}} &
\betaCB  &=\sqrt[X]{C^{Z-2X}-B^{Y-2X}}
\label{eqn:60b}
\end{align}
\end{subequations}
Per \cref{Thm:2.11_Coprime_Alpha_Beta_Irrational}, $\alphaCA$, $\alphaCB$, $\betaCA$, and $\betaCB$ as defined in \cref{eqn:60a,eqn:60b} are irrational when gcd$(A,B,C)=1$.  Thus $\MCB$ and $\MCA$ in \cref{eqn:59} are irrational.

Therefore as consequences of both cases 1 and 2, slopes $\MCB$ and $\MCA$ must be irrational when $A^X \neq B^Y$ and gcd$(A,B,C)=1$.

\bigskip

As a side note, an alternate way to define $\alphaCA$, $\alphaCB$, $\betaCA$, and $\betaCB$ is to select any real value for $\alphaCA$ and $\alphaCB$ and then derive $\betaCA$ and $\betaCB$ or select any real value for $\betaCA$ and $\betaCB$ and then derive $\alphaCA$ and $\alphaCB$, per \cref{Thm:2.9_Real_Alpha_Beta}.  However per \cref{Thm:2.13_Coprime_Any_Alpha_Beta_Irrational_Indeterminate}, when gcd$(A,B,C)=1$, the derived values of $\alphaCA$, $\alphaCB$, $\betaCA$, and $\betaCB$ are either irrational or indeterminate.  Thus $\MCB$ and $\MCA$ in \cref{eqn:59} are irrational or indeterminate.  If the slopes $\MCB$ and $\MCA$ were rational, then $\alpha$ and $\beta$ would both be determinate, thus a contradiction.  Hence here too the requirements to ensure $\MCB$ and $\MCA$  are rational cannot be met when gcd$(A,B,C)=1$.

\bigskip

\textbf{Scenario 2 of 2:  $\bm{\displaystyle{A^X=B^Y}$}.} Suppose $A^X=B^Y$.  Given the bases are reduced, we conclude that $A=B$.  Per this theorem, gcd$(A,B)=1$ and thus $A\neq B$, hence this scenario is impossible.

\bigskip

Since the scenarios are mutually exclusive and exhaustive given gcd$(A,B)=1$, and given that in each scenario, the slopes $\MCB$ and $\MCA$ are irrational, then the line that goes through the origin and through the integer solution must also have irrational slopes.  However, as already proven in \cref{Thm:2.1_Irrational_Slope_No_Lattice}, lines through the origin with irrational slopes cannot go through any non-trivial lattice points.  \Conjecture, if gcd$(A,B,C) = 1$, then slopes $\MCB$ and $\MCA$ are irrational and thus there is no integer solution.
\end{proof}

We know that each grid point $(A,B,C)$ subtends a line through the origin and that point, whereby that point is supposed to be an integer solution that satisfies the conjecture.  We also know that each grid point corresponds to a set of slopes.  Further, we know from \cref{Thm:2.1_Irrational_Slope_No_Lattice} that a line through the origin with an irrational slope does not pass through any non-trivial lattice points.  Since both $\MCB$ and  $\MCA$ are irrational when gcd$(A,B,C)=1$, then their corresponding lines fail to go through any non-trivial lattice points, and thus these slopes mean there are no corresponding integer solutions for $A$, $B$, and $C$.  Hence there is no integer solution satisfying the conjecture when gcd$(A,B,C)=1$.
\label{Section:Impossibility_End}

\bigskip

\Line
\subsection{Requirement for Possibility of the Terms}
\label{Section:Possibility_Start}
Having established that the slopes of a line through the origin and the lattice point $(A,B,C)$ are irrational when gcd$(A,B,C)=1$ translates to no integer solutions to the conjecture and the non-existence of a non-trivial lattice point.  We now consider the reverse; if there is an integer solution that satisfies the conjecture, then gcd$(A,B,C)$ must be greater than 1 translates to the existence of a non-trivial lattice point through which the line passes.

\bigskip

\begin{theorem}
\label{Thm:2.15_Main_Proof_Solutions_Then_Not_Coprime}
\Conjecture, if there are integer solutions, then $gcd(A,B,C)>1$.
\end{theorem}

\begin{proof}
Given the configuration of the conjecture, $A$, $B$, $C$, $X$, $Y$, and $Z$ are positive integers, and thus $A^X$, $B^Y$, and $C^Z$ are also integers.  A set of values that satisfy the conjecture correspond to a point on a scatter plot with axes $A$, $B$, and $C$.  See \cref{Fig:3DScatterMap}.  Based on \cref{eqn:1} the line passing through the origin and the point $(A,B,\sqrt[Z]{A^X+B^Y})$ can be expressed based on its slopes in the three planes (see \cref{eqn:55a,eqn:55b,eqn:56a} and \cref{Fig:ScatterPlotWithAngles2}), namely
\begin{subequations}
\begin{align}
\MCB &= \frac{\sqrt[Z]{A^X+B^Y}}{B} =\frac{C}{B}   &
\MCA &= \frac{\sqrt[Z]{A^X+B^Y}}{A} =\frac{C}{A}   \label{eqn:61a} \\
\MCB &= \sqrt[Z]{\frac{A^X+B^Y}{B^Z}} &
\MCA &= \sqrt[Z]{\frac{A^X+B^Y}{A^Z}}              \label{eqn:61b}
\end{align}
\end{subequations}
where $\MCB$ and $\MCA$ are the slopes of the lines through the origin and the given point in the $C \times B$ and $C \times A$ planes, respectively.  We can likewise define $\MBA$ as the slope of the lines through the origin and the given point in the $A\times B$ plane based on \cref{eqn:55c}, namely
\begin{equation}
  \MBA  = \frac{B}{A} \label{eqn:62}
\end{equation}
We know from \cref{Thm:2.14_Main_Proof_Coprime_No_Solutions} that $\MCB$ and $\MCA$ cannot be rational if gcd$(A,B,C)=1$.  Suppose gcd$(A,B,C)=k$ where $k\geq 2$.  Thus with integer $k$ common to $A$, $B$, and $C$, we can express \cref{eqn:1} with the common term, namely
\begin{equation}
a^Xk^X + b^Yk^Y = c^Zk^Z \label{eqn:63}
\end{equation}
where $a$, $b$, and $c$ are positive coprime integer factors of $A$, $B$, and $C$ respectively, and where $A=ak$, $B=bk$, $C=ck$, and where $A^X=a^Xk^X$, $B^Y=b^Yk^Y$, and $C^Z=c^Zk^Z$.  We can thus express the slopes from \cref{eqn:61b,eqn:62} with the common term, namely
\begin{subequations}
\begin{gather}
\MCB  = \sqrt[Z]{\frac{a^Xk^X+b^Yk^Y}{b^Zk^Z}} \qquad\qquad
\MCA  = \sqrt[Z]{\frac{a^Xk^X+b^Yk^Y}{a^Zk^Z}} \label{eqn:64a}\\
\MBA  = \frac{bk}{ak} = \frac{b}{a}            \label{eqn:64b}
\end{gather}
\end{subequations}
We observe that the common term $k$ cancels from $\MBA$ in \cref{eqn:64b} and thus $\MBA$ is rational regardless of the common term.  We can simplify \cref{eqn:64a} as
\begin{equation}
\MCB  = \left(\frac{(ak)^X}{(bk)^Z} +(bk)^{Y-Z}\right) ^{\frac{1}{Z}}   \qquad  \MCA  =  \left((ak)^{X-Z} + \frac{(bk)^Y}{(ak)^Z} \right) ^{\frac{1}{Z}}    \label{eqn:65}
\end{equation}

Before applying Newton's generalized binomial expansion to slopes $\MCB$ and $\MCA$ in \cref{eqn:65}, we must first consider two mutually exclusive scenarios:

\bigskip

\textbf{Scenario 1 of 2:  $\bm{\displaystyle{A^X \neq B^Y}$}.}  Suppose $A^X\neq B^Y$.  As such
$\displaystyle{\frac{A^X}{B^Z}\neq  B^{Y-Z}}$ and $\displaystyle{A^{X-Z} \neq \frac{B^Y}{A^Z}}$, and thus $\displaystyle{\frac{(ak)^X}{(bk)^Z}\neq (bk)^{Y-Z}}$ and $\displaystyle{(ak)^{X-Z} \neq \frac{(bk)^Y}{(ak)^Z}}$.  Therefore we can re-express $\MCB$ and $\MCA$ in \cref{eqn:65} as
\begin{subequations}
\begin{align}
m_{_{C, B}} &= \left[\frac{(ak)^X}{(bk)^Z}\right]^{\frac{1}{Z}}
               \left[1+\frac{(bk)^{Y-Z}} {\left(\frac{(ak)^X}{(bk)^Z}\right)}\right]
               ^{\frac{1}{Z}}   &
m_{_{C, A}} &= \left[(ak)^{X-Z}\right]^{\frac{1}{Z}} \left[1+
               \frac{ \left( \frac{(bk)^Y}{(ak)^Z} \right)} {(ak)^{X-Z}} \right] ^{\frac{1}{Z}}
\label{eqn:66a} \\
m_{_{C, B}} &= \frac{(ak)^{\frac{X}{Z}}}{bk}  \left(1 +\frac{(bk)^Y}{(ak)^Z}\right)
              ^{\frac{1}{Z}}   &
m_{_{C, A}} &= {(ak)^{\frac{X}{Z}-1}} \left(1 +\frac{(bk)^Y}{(ak)^Z}\right) ^{\frac{1}{Z}}
\label{eqn:66b}
\end{align}
\end{subequations}

\bigskip

Given there are integer solutions that satisfy the conjecture, then $\MCB$ and $\MCA$ must be rational.  Based on \cref{eqn:66b}, there are two cases that can ensure both $\MCB$ and $\MCA$ are rational:
\begin{enumerate}
\item The term $\displaystyle{\left(1+\frac{(bk)^Y}{(ak)^Z}\right)^{\frac{1}{Z}}}$ from
      \cref{eqn:66b} is rational, therefore both terms
      $\displaystyle{\frac{(ak)^{\frac{X}{Z}}}{bk}}$ and $(ak)^{\frac{X}{Z}-1}$ must be
      rational so their respective products are rational.
\item The term $\displaystyle{\left(1+\frac{(bk)^Y}{(ak)^Z}\right)^{\frac{1}{Z}}}$ from
      \cref{eqn:66b} is irrational, therefore both terms
      $\displaystyle{\frac{(ak)^{\frac{X}{Z}}}{bk}}$ and $(ak)^{\frac{X}{Z}-1}$ are irrational.
      However the irrationality of these two terms are canceled with the irrationality
      of the denominator of $\displaystyle{\left(1+\frac{(bk)^Y}{(ak)^Z}\right)^{\frac{1}{Z}}}$
      so their respective products are rational.
\end{enumerate}

\bigskip

\noindent\textbf{Case 1 of 2: the terms of $\bm{\MCB }$ and $\bm{\MCA }$ are rational}

Starting with the first case, assume $\displaystyle{\left(1+ \frac{(bk)^Y}{(ak)^X}\right)^{\frac{1}{Z}}}$ in \cref{eqn:66b} is rational and thus $\displaystyle{\frac{(ak)^{\frac{X}{Z}}}{bk}}$ and $(ak)^{\frac{X}{Z}-1}$ are rational.  Applying the binomial expansion to \cref{eqn:66b} gives us

\begin{equation}
\MCB =  \frac{(ak)^{\frac{X}{Z}}}{bk}  \sum \limits_{i=0}^{\infty}
                \binom{\frac{1}{Z}}{i} \frac{(bk)^{Yi}}{(ak)^{Xi}}  \qquad
\MCA =  {(ak)^{\frac{X}{Z}-1}} \sum \limits_{i=0}^{\infty}
            \binom{\frac{1}{Z}}{i} \frac{(bk)^{Yi}}{(ak)^{Xi}}
\label{eqn:67}
\end{equation}
The binomial coefficient $\displaystyle{\binom{\frac{1}{Z}}{i}}$ and ratio $\displaystyle{\frac{(bk)^{Yi}}{(ak)^{Xi}}}$ in both formulas in \cref{eqn:67} are rational for all $i$, as are their products, regardless of the value of $k$.  Hence we consider the terms $\displaystyle{\frac{(ak)^{\frac{X}{Z}}}{bk}}$ and $(ak)^{\frac{X}{Z}-1}$ in \cref{eqn:67}.  Since integers $A$ and $ak$ are reduced (not perfect powers), there are only three possibilities that ensure they are rational:
\begin{enumerate}
\item $a$ is a perfect power of $Z$, or
\item $X$ is an integer multiple of $Z$, or
\item $k > 1$ such that $ak$ is a perfect power of $Z$.
\end{enumerate}

Suppose $a$ is a perfect power of Z.  Since $A=ak$ cannot be a perfect power of $Z$ given that $A$ is reduced, then $k$ must be greater than 1 so that when factored from composite $A$, the remaining factors of $A$ are a perfect power of $Z$.  Hence $k>1$ and thus gcd$(A,B,C)>1$.

Suppose instead $X$ is an integer multiple of $Z$, hence exponent $X=iZ$ for some positive integer $i$.  However, per \cref{Dfn:2.1_X_cannot_be_mult_of_Z} on page \pageref{Dfn:2.1_X_cannot_be_mult_of_Z}, $X$ is not an integer multiple of $Z$.  Thus if $k=1$, the term $\displaystyle{(ak)^{\frac{X}{Z}}=A^{\frac{X}{Z}}=a^{\frac{X}{Z}}}$ is not a perfect power and thus $\MCB$ and $\MCA$ are irrational, contradicting their given rationality.  If instead $k$ is a multiple of $a$ such as $k=a^j$ for some positive integer $j$, then $\displaystyle{(ak)^{\frac{X}{Z}}=a^{\frac{X+j}{Z}}}$ for which $X+j$ could be a multiple of $Z$.  Thus $\MCB$ and $\MCA$ are rational only for some values $k>1$ and thus gcd$(A,B,C)>1$.

\bigskip

\noindent\textbf{Case 2 of 2: the terms of $\bm{\MCB }$ and $\bm{\MCA }$ are irrational}

Considering the second case, assume $\displaystyle{\left(1+ \frac{(bk)^Y}{(ak)^X}\right)^{\frac{1}{Z}}}$ in \cref{eqn:66b} is irrational, and thus $\displaystyle{\frac{(ak)^{\frac{X}{Z}}}{bk}}$ and $(ak)^{\frac{X}{Z}-1}$ must also be irrational such that they cancel the irrationality with multiplication.  Since slopes $\MCB$ and $\MCA$ in \cref{eqn:66b} are rational with irrational terms, both slopes in \cref{eqn:64a} must be rational.

We can re-express slopes $\MCB$ and $\MCA$ from \cref{eqn:61a} as
\begin{equation}
\MCB = \frac{C}{\sqrt[Y]{C^Z-A^X}}   \qquad\qquad\quad
\MCA = \frac{C}{\sqrt[X]{C^Z-B^Y}}
\label{eqn:68}
\end{equation}
in which the denominators must both be rational.  Per \cref{Thm:2.8_Functional_Form}, the denominators in \cref{eqn:68} can be reparameterized and thus \cref{eqn:68} becomes
\begin{equation}
\MCB = \frac{C}{(C+A)M_{_1}\alphaCA  -CAM_{_1}\betaCA } \qquad\qquad
\MCA = \frac{C}{(C+B)M_{_2}\alphaCB  -CBM_{_2}\betaCB }
\label{eqn:69}
\end{equation}
where $\alphaCA$, $\alphaCB$, $\betaCA$, and $\betaCB$ are positive rational numbers and $M_{_1}$ and $M_{_2}$ are positive scalars.   Per \cref{Thm:2.8_Functional_Form,Thm:2.9_Real_Alpha_Beta} $\alphaCA$, $\alphaCB$, $\betaCA$ and $\betaCB$ are defined as
\begin{subequations}
\begin{align}
\alphaCA &=\sqrt[Y]{C^{Z-Y}-A^{X-Y}} &
\alphaCB &=\sqrt[X]{C^{Z-X}-B^{Y-X}}
\label{eqn:70a} \\
\betaCA  &=\sqrt[Y]{C^{Z-2Y}-A^{X-2Y}} &
\betaCB  &=\sqrt[X]{C^{Z-2X}-B^{Y-2X}}
\label{eqn:70b}
\end{align}
\end{subequations}
Per \cref{Thm:2.11_Coprime_Alpha_Beta_Irrational}, $\alphaCA$, $\alphaCB$, $\betaCA$, and $\betaCB$ as defined in \cref{eqn:70a,eqn:70b} are irrational when gcd$(A,B,C)=k=1$.  However, since $\MCB$ and $\MCA$ in \cref{eqn:70a,eqn:70b} are rational, we need to consider the common factor in the bases.

Per \cref{Thm:2.12_Rational_Alpha_Beta_Rational_Then_Not_Coprime}, $\displaystyle{\sqrt[Y]{C^{Z-Y}-A^{X-Y}}}$ and $\displaystyle{\sqrt[Y]{C^{Z-2Y}-A^{X-2Y}}}$ are rational only if gcd$(A,B,C)>1$.  By extension, $\displaystyle{\sqrt[X]{C^{Z-X}-B^{Y-X}}}$ and $\displaystyle{\sqrt[X]{C^{Z-2X}-B^{Y-2X}}}$ are also rational only if gcd$(A,B,C)>1$.

Since slopes $\MCB$ and $\MCA$ in \cref{eqn:69} are rational, then their denominators are rational, and thus $\alphaCA$, $\alphaCB$, $\betaCA$, and $\betaCB$ must be rational.  Per the defintions of $\alphaCA$, $\alphaCB$, $\betaCA$, and $\betaCB$ in \cref{eqn:70a,eqn:70b} and given \cref{Thm:2.12_Rational_Alpha_Beta_Rational_Then_Not_Coprime} in which these terms can only be rational when gcd$(A,B,C)>1$, we conclude $k$ must be greater than 1 and thus gcd$(A,B,C)>1$.

\bigskip

Therefore as consequences of both cases 1 and 2, for slopes $\MCB$ and $\MCA$ to be rational when $A^X\neq B^Y$ requires gcd$(A,B,C)>1$, and thus $A$, $B$, and $C$ must share a common factor greater than 1.

\bigskip

\textbf{Scenario 2 of 2:  $\bm{\displaystyle{A^X=B^Y}$}.}  If $A^X=B^Y$, then gcd$(A,B)=k>1$.  Hence by definition, $k$ is a factor of $A^X+B^Y$ and of $C^Z$, and thus $A$, $B$, and $C$ must share a common factor greater than 1.

\bigskip

\Conjecture, when there exist integer solutions, two scenarios show common factor $k$ must be greater than 1 to ensure slopes $\MCB$ and $\MCA$ are rational.  We know that each grid point $(A,B,C)$ subtends a line through the origin and that point, whereby that point is supposed to be an integer solution that satisfies the conjecture.  We also know that each grid point corresponds to a set of slopes.  Further, we know from \cref{Thm:2.1_Irrational_Slope_No_Lattice} that a line through the origin with an irrational slope does not pass through any non-trivial lattice points.  Since both $\MCB$ and  $\MCA$ are rational only for some common factor $k>1$, then gcd$(A,B,C)=k$ is required.  We know $\MBA$ is always rational but since $\MCB$ and $\MCA$ can be rational only when gcd$(A,B,C)>1$ for only certain common factors, then we know the lines go through non-trivial lattice points, and thus these slopes mean there can be integer solutions for $A$, $B$, and $C$.  Hence there can be integer solutions satisfying the conjecture only when gcd$(A,B,C)>1$.
\end{proof}
\label{Section:Possibility_End}

\bigskip

\section{Conclusion}
Every set of values that satisfy the Tijdeman-Zagier conjecture corresponds to a lattice point on a multi-dimensional Cartesian grid. Together with the origin this point defines a line in multi-dimensional space.  This line requires a rational slope in order for it to pass through a non-trivial lattice point.  Hence the core of the various proofs contained herein center on the irrationality of the slope based on the coprimality of the terms.  Several key steps were required to establish this relationship and then support the proof.

\cref{Thm:2.2_Coprime,Thm:2.3_Coprime,Thm:2.4_Coprime} establish that within the relation \BealsEq\, if any pair of terms, $A$, $B$, and $C$ is coprime, then all 3 terms must be coprime, and if all 3 terms are coprime, then each pair of terms must likewise be coprime.    Likewise, \cref{Thm:2.5_X_cannot_be_mult_of_Z} establish a similarly restrictive relationship between the exponents, namely that  exponents $X$ and $Y$ cannot be integer multiples or unit fractions of exponent $Z$ and that $Z$ cannot be an integer multiple of $X$ or $Y$.

\cref{Thm:2.6_Initial_Expansion_of_Differences,Thm:2.7_Indeterminate_Limit} establish that the difference of powers can be factored and expanded based on an arbitrary and indeterminate upper limit.

\cref{Thm:2.8_Functional_Form,Thm:2.9_Real_Alpha_Beta} establish that $A^X$ could be parameterized as a linear combination of $C+B$ and $CB$, with two parameters.  \cref{Thm:2.10_No_Solution_Alpha_Beta_Irrational,Thm:2.11_Coprime_Alpha_Beta_Irrational,Thm:2.13_Coprime_Any_Alpha_Beta_Irrational_Indeterminate} establish that when these parameters are irrational there can be no integer solution satisfying the conjecture and if gcd$(A,B,C)=1$, then these parameters must be irrational.  \cref{Thm:2.12_Rational_Alpha_Beta_Rational_Then_Not_Coprime} establishes that if
these two parameters are rational, then gcd$(A,B,C)>1$.

The relationships between coprimality of terms and irrationality of the parameters
(\cref{Thm:2.8_Functional_Form,Thm:2.9_Real_Alpha_Beta,Thm:2.10_No_Solution_Alpha_Beta_Irrational,Thm:2.11_Coprime_Alpha_Beta_Irrational,Thm:2.12_Rational_Alpha_Beta_Rational_Then_Not_Coprime,Thm:2.13_Coprime_Any_Alpha_Beta_Irrational_Indeterminate}) are critical to the slopes that are core to the remaining theorems.  It is shown that the slopes are functions of these parameters and thus the irrationality properties of the parameters translate to irrationality conditions for the slopes.

\cref{Thm:2.1_Irrational_Slope_No_Lattice} establishes that a line with an irrational slope that passes through the origin will not pass through any non-trivial lattice points.  This simple, subtle theorem is critical to the proof since the link between irrationality of slope and non-integer solutions is key to relating the outcomes to coprimality of terms.  The logic of the proof is that integer solutions which satisfy the conjecture can be expressed only with a set of rational slopes and thus tests of the slope rationality are equivalent to tests of the integrality of the solution.

\cref{Thm:2.14_Main_Proof_Coprime_No_Solutions} establishes that when gcd$(A,B,C)=1$, the slopes are irrational.  Thus if the slopes are irrational, then the line that is equivalent to the integer solution does not pass through non-trivial lattice points, hence there is no integer solution.  \cref{Thm:2.15_Main_Proof_Solutions_Then_Not_Coprime} establishes the reverse, namely that the slopes of the corresponding lines can only be rational when gcd$(A,B,C)>1$, and that integer solutions satisfying the conjecture fall on the lines with rational slopes.

Any proof of the Tijdeman-Zagier conjecture requires four conditions be satisfied:
\begin{itemize}
\item $A$, $B$, $C$, $X$, $Y$, and $Z$ are positive integers.
\item $X,Y,Z\geq3$
\item \BealsEq
\item gcd$(A,B,C)=1$
\end{itemize}
Since the set of values that satisfy the conjecture is directly a function of rationality of slopes, we have demonstrated the explicit linkage between the coprimality aspect of the conjecture, the integer requirement of the framework, and properties of slopes of lines through the origin.  Via contradiction these theorems prove the four conditions cannot be simultaneously met.  Given the fully exhaustive and mutual exclusivity of the theorems, the totality of the conjecture is thus proven.

\bigskip

\section*{Acknowledgment}
The authors acknowledge and thank emeritus Professor Harry Hauser for guidance and support in the shaping, wordsmithing, and expounding the theorems, proofs, and underlying flow of the document, and for the tremendous array of useful suggestions throughout.


\section*{References}

\end{document}